\documentclass[leqno,12pt]{amsart} 
\usepackage{amssymb} 
 
\theoremstyle{plain} 
\newtheorem{theorem}{\indent\sc Theorem}[section] 
\newtheorem{lemma}[theorem]{\indent\sc Lemma}
\newtheorem{corollary}[theorem]{\indent\sc Corollary}
\newtheorem{proposition}[theorem]{\indent\sc Proposition}
\newtheorem{claim}[theorem]{\indent\sc Claim}
\theoremstyle{definition} 
\newtheorem{definition}[theorem]{\indent\sc Definition}
\newtheorem{remark}[theorem]{\indent\sc Remark}

\begin{document}

\title[Ricci curvature]{A weakly second order differential structure on rectifiable metric measure spaces} 

\author[Shouhei Honda]{Shouhei Honda} 



\address{ 
Faculty of Mathmatics \endgraf
Kyushu University \endgraf 
744, Motooka, Nishi-Ku, \endgraf 
Fukuoka 819-0395 \endgraf 
Japan
}
\email{honda@math.kyushu-u.ac.jp}

\maketitle

\begin{abstract}
We give the definition of angles on a Gromov-Hausdorff limit space of a sequence of complete $n$-dimensional Riemannian manifolds with a lower Ricci curvature bound. We apply this to prove there is a weakly second order differential structure on these spaces and prove there is a unique
Levi-Civita connection allowing us to define the Hessian of a twice differentiable function.
\end{abstract}

\section{Introduction}
Let $X$ be a metric space. We say that a map $\gamma$ from $[0, l]$ to $X$ is a \textit{minimal geodesic} if $\gamma$ is an isometric embedding.
Let $\gamma_1$ and $\gamma_2$ be minimal geodesics on $X$ beginning at a point $x \in X$.
Define the angle $\angle \dot{\gamma_1} \dot{\gamma_2} \in [0, \pi]$ between $\gamma_1$ and $\gamma_2$ at $x$ by
\begin{equation}
\cos \angle \dot{\gamma_1} \dot{\gamma_2} = \lim_{t \to 0}\frac{2t^2-\overline{\gamma_1(t), \gamma_2(t)}^2}{2t^2}
\end{equation}  
if the limit exists, where $\overline{x,y}$ is the distance between $x$ and $y$.

This notion of an angle is crucial in the study metric spaces.
For example, on Alexandrov spaces (or CAT($\kappa$)-spaces), the angle between every two minimal geodesics beginning at a common point always exists.
The existence directly follows from some monotonicity property induced by Toponogov's comparison inequality. 
Roughly speaking, the monotonicity property is closely related to a lower (or upper) bound of sectional curvature of the space.
See for instance a fundamental work about Alexandrov spaces \cite{bgp} by Burago-Gromov-Perelman.
Note that in general, angles are \textit{not} well-defined.

Now we consider the following question:

\textbf{Question}: \textit{is the angle between two given minimal geodesics beginning at a common point on a metric (measure) space with a lower Ricci curvature bound well-defined?}

Since the angle on Alexandrov space is well-defined, the answer to the question is affirmative under a lower bound of sectional curvature.
There are many important works for a lower Ricci curvature bounds on metric measure spaces.
See for instance \cite{oh, st1, st2, lo-vi, vi} by Lott-Villani, Ohta and Sturm.
Note that a typical example of them is a Gromov-Hausdorff limit space of a sequence of Riemannian manifolds with a lower Ricci curvature bound.

In this paper we prove that on limits of manifolds with lower Ricci curvature bounds, the answer is \textit{ALMOST POSITIVE}, in a way which we will soon make more precise in Theorem \ref{angle}.
First we observe that the following Colding-Naber's recent very interesting result implies that in general, an answer to the question above is \textit{NEGATIVE}:
\begin{theorem}\cite[Theorems $1.2$ and $1.3$]{co-na}\label{cn}
For every $n \ge 3$, there exists a pointed proper metric space $(Y, p)$ with the following properties:
\begin{enumerate}
\item $(Y, p)$ is a noncollapsing Gromov-Hausdorff limit space of a sequence of pointed $n$-dimensional complete Riemannian manifolds with a lower Ricci curvature bound.
\item All points of $Y$ are regular points. Moreover,  $Y$ is a uniform Reifenberg space.
\item For every two minimal geodesics $\gamma_1$, $\gamma_2$ beginning at $p$ and every $\theta \in [0, \pi]$, there exists a sequence $t_i \to 0$ such that 
\begin{equation}
\cos \theta = \lim_{i \to \infty}\frac{2t_i^2-\overline{\gamma_1(t_i), \gamma_2(t_i)}^2}{2t_i^2}.
\end{equation}
\end{enumerate}
\end{theorem}
Note that $(Y, p)$ as above has some nice properties.
See \cite[Theorem $1.2$]{co-na} by Colding-Naber for the details.
This important example implies that even on a metric space with a lower Ricci curvature bound and nice properties, in general, angles are not well-defined.

In order to give the first main theorem of this paper,
let $(M_{\infty}, m_{\infty})$ be a Gromov-Hausdorff limit space of a sequence of pointed complete $n$-dimensional Riemannian manifolds  $\{(M_i, m_i)\}_{i<\infty}$ with $\mathrm{Ric}_{M_i} \ge -(n-1)$.
See \cite{ch-co, ch-co1, ch-co2, ch-co3} for the wonderful structure theory of $M_{\infty}$ developed by Cheeger-Colding.
The following is the first main result of this paper:
\begin{theorem}\label{angle}
Let $p, q \in M_{\infty} \setminus \{m_{\infty}\}$ with $m_{\infty} \not \in C_p \cup C_q$.
Then the angle $\angle pm_{\infty}q \in [0, \pi]$ of $pm_{\infty}q$ is well-defined.
In fact, we have 
\begin{equation}
\cos \angle pm_{\infty}q=\lim_{t \to 0}\frac{2t^2-\overline{\gamma_p(t), \gamma_q(t)}^2}{2t^2}
\end{equation}
for any minimal geodesics $\gamma_p$ from $m_{\infty}$ to $p$ and $\gamma_q$ from $m_{\infty}$ to $q$, where $C_p$ is the cut locus of $p$ defined by $C_p =\{x \in M_{\infty}; \overline{p, x} + \overline{x, z}> \overline{p, z}$ for every $z \in M_{\infty} \setminus \{x\}\}.$
\end{theorem}
Theorem \ref{angle} implies that 
\begin{corollary}\label{angle1}
The angle between any pair of minimal geodesics, $\gamma_i: [0, l_i] \to M_{\infty}$ beginning at $x \in M_{\infty}$ is well-defined as long as they can be extended minimally through $x$, $\exists \epsilon >0$ such that $\gamma_i:[-\epsilon, \epsilon] \to M_{\infty}$ is minimal.
\end{corollary}
In \cite[Theorem $3.2$]{ho}, the cut locus is shown to have measure zero with respect to any limit measure, $\upsilon$, on $M_{\infty}$,
\begin{corollary}\label{angle2}
For every $p, q \in M_{\infty}$, the angle $\angle pxq$ is well-defined for $\upsilon$-almost every $x \in M_{\infty}$.
\end{corollary}
See  Theorem \ref{well3} for the proof of Theorem \ref{angle}. 
See also \cite[Corollary A.$4$]{co-na1} by Colding-Naber.

We will also discuss some H$\ddot{\mathrm{o}}$lder continuity of angles (Corollary \ref{hold}) and show the existence of the \textit{weakly $C^{1, \alpha}$-structure} on $M_{\infty}$ in some sense (Corollary \ref{holder}) for some $\alpha=\alpha(n)<1$. 

The second main result of this paper is the following:
\begin{theorem}\label{limit}
$M_{\infty}$ has a weakly second order differential structure.
\end{theorem}
See Definition \ref{2nd2nd2nd} for the precise definition of a weakly second order differential structure on metric (measure) spaces.
Note that this second order differential structure is \textit{better} than the $C^{1, \alpha}$-structure above in some sense.
In fact, for instance, we can give a suitable definition of twice differentiable functions on a space having a weakly second order differential structure (Definition \ref{2ndd}).
We will show that all eigenfunctions with respect to the Dirichlet problem on $M_{\infty}$ are weakly twice differentiable (Corollary \ref{eigen}).

On the other hand, Cheeger defined a notion of a weak Riemannian metric in Section $4$ of \cite{ch1} which is known to be well-defined on $M_{\infty}$ by Section $7$ of \cite{ch-co3} by Cheeger-Colding.
We will review the definition of this notion in Theorems \ref{29292} and \ref{harm3}.
We will show that the Riemannian metric of $M_{\infty}$ is \textit{Lipschitz} in some sense with respect to a weakly second order differential structure as in Theorem \ref{limit}. See Theorem \ref{harm}.
As corollaries, we will show that the \textit{Levi-Civita connection on $M_{\infty}$} exists uniquely (Theorem \ref{levi3}), and study the \textit{Hessian of a twice differentiable function} (Proposition \ref{hess2}).  

For example, let $(Z, z)$ be a noncollapsing Gromov-Hausdorff limit of a sequence of pointed complete $n$-dimensional Einstein manifolds $\{(\hat{M}_i, \hat{m}_i)\}_i$ with $\mathrm{Ric}_{\hat{M}_i}=H(n-1)$,  where $H$ is a fixed real number.
Then in \cite{ch-co1} Cheeger-Colding showed that the regular set $\mathcal{R}$ of $Z$ is open and a smooth Riemannian manifold.
We see that the Levi-Civita connection given in this paper coincides with that defined by the smooth structure of $\mathcal{R}$.
See Theorems \ref{harm3}, \ref{levi3}, \ref{harm} and \cite[Theorem $7.3$]{ch-co1} for the details. 

Next we give a remark about Theorem \ref{limit}.
For that, we now recall a celebrated work for \textit{(measurable) differentiable structure} on metric measure spaces by Cheeger. 
In \cite{ch1}, Cheeger showed that a metric measure space satisfying the Poincar\'e inequality and doubling condition has a differentiable structure in some sense.
For instance, we can also find very interesting examples of them in \cite{laa, pans} by Laakso and Pansu.
See also \cite{keith1} by Keith.
Note that $M_{\infty}$ with a limit measure is a typical example of them. 
It is important that we can discuss the once differentiability for functions on such metric measure spaces.
In fact, it is shown that all Lipschitz functions on such spaces are differentiable almost everywhere in some sense, as in Rademacher's theorem \cite{rad}. 
On the other hand, in general, it seems that it is not easy to give a suitable definition of a second order differential structure on metric measure spaces.
However, in several situations, e.g., Alexandrov spaces, we can consider such a second order differential structure (see for instance \cite{bgp, ot, ot-sh, pere, pere1} by Burago-Gromov-Perelman, Otsu, Otsu-Shioya and Perelman). 
The notion of weakly second order differential structure on metric measure spaces given in this paper gives such a framework including limit spaces of Riemannian manifolds with a lower Ricci curvature bound.

Finally we introduce fundamental tools used in the proofs of Theorems \ref{angle} and \ref{limit}. 
In the proof of Theorem \ref{angle}, we will essentially use the proof of Cheeger-Colding's splitting theorem \cite[Theorem $6.64$]{ch-co} and several fundamental properties of the convergence of the differentials of Lipschitz functions with respect to the measured Gromov-Hausdorff topology given in \cite{Ho} by the author.
In the proof of Theorem \ref{limit}, we will essentially use several fundamental properties of the convergence of spectral structures with respect to the measured Gromov-Hausdorff topology given in \cite{KS} by Kuwae-Shioya and also use several results given in \cite{Ho} again.

As a continuation of this paper, in \cite{Ho5} we will prove a Bochner-type inequality on $M_{\infty}$ which keeps the term of Hessian defined in Section $3$, discuss a weak $L^2$-convergence of Hessians with respect to the Gromov-Hausdorff topology and give a relationship between the Laplacian defined by using the twice differential structure in Section $3$ and the Dirichlet Laplacian defined by Cheeger-Colding in \cite{ch-co3}.
In particular, we will show that in noncollapsing setting, these Laplacians coincide on a dense subspace in $L^2$.

The organization of this paper is as follows:

In Section $2$, we will introduce several fundamental notions on metric measure spaces and on limit spaces of Riemannian manifolds.
In Section $3$, we will give the definition of the weakly second order differential structure on metric measure spaces and study several properties.
In Section $4$, we will give the proofs of Theorems \ref{angle} and \ref{limit}.
 
\textbf{Acknowledgments.}
The author would like to express my appreciation to Tobias Holck Colding and Aaron Naber for helpful comments.
He is grateful to Kazuhiro Kuwae for giving valuable suggestions.
He wishes to thank Ayato Mitsuishi, Koichi Nagano, Takashi Shioya and Takao Yamaguchi, for giving valuable comments at the informal geometry seminar at Tsukuba University.
He was supported  by Grant-in-Aid for Research Activity Start-up $22840027$ from JSPS and Grant-in-Aid for Young Scientists (B) $24740046$.
\section{Preliminaries}
In this section, we will introduce several fundamental notions on metric measure spaces and on limit spaces of Riemannian manifolds.
Let $X$ be a metric space.
For $R>0$, $x \in X$, we set $B_R(x)=\{y \in X; \overline{x,y}<R\}$ and $\overline{B}_R(x)=\{y \in X; \overline{x,y}\le R\}$.
\subsection{Metric measure spaces}
We say that $X$ is \textit{proper} if every bounded closed subset of $X$ is compact.
We say that $X$ is a \textit{geodesic space} if for every $x, y \in X$, there exists a minimal geodesic $\gamma$ from $x$ to $y$.
Let $\upsilon$ be a Radon measure on $X$.
In this paper, we say that $(X, \upsilon)$ is a \textit{metric measure space} if $X$ is a proper geodesic space and if $\upsilon(B_r(x))>0$ for every $x \in X$ and every $r>0$.
We now recall the notion of \textit{rectifiability} for metric measure spaces defined by Cheeger-Colding in \cite{ch-co3}:
\begin{definition}\cite[Definition $5.3$]{ch-co3}\label{rectifiable}
Let $(X, \upsilon)$ be a metric measure space.
We say that \textit{$X$ is weakly $\upsilon$-rectifiable (or $(X, \upsilon)$ is weakly rectifiable)} if there exist a positive integer $m$, a collection of Borel subsets $\{C_{i}^l\}_{1 \le l \le m, i \in \mathbf{N}}$ of $X$, and a collection of bi-Lipschitz embedding maps $\{ \phi_{i}^l: C_{i}^l \rightarrow \mathbf{R}^l \}_{l, i}$ with the following properties $(1)$ and $(2)$:
\begin{enumerate}
\item $\upsilon(X \setminus \bigcup_{l,i}C_{i}^l)=0$.
\item $\upsilon$ is Ahlfors $l$-regular at every $x \in C_{i}^l$, i.e., there exist $C \ge 1$ and $r>0$ such that 
$C^{-1} \le \upsilon (B_t(x)) / t^l \le C$ 
for every $0<t<r$.
\end{enumerate}
Moreover we say that \textit{$X$ is $\upsilon$-rectifiable (or $(X, \upsilon)$ is rectifiable)} if the following condition holds: 
\begin{enumerate}
\item[(3)] For every $l$, every $x \in \bigcup_{i \in \mathbf{N}}C_{i}^l$ and every $0 < \delta < 1$, there exists $i$ such that $x \in C_{i}^l$ and that 
a map $\phi_{i}^l$ is $(1 \pm \delta)$-bi-Lipschitz to the image $\phi_{i}^l(C_{i}^l)$.
\end{enumerate}
\end{definition}
Our third condition is a strong additional condition not usually required in the definition of rectifiable spaces.
See \cite[Definition $5.3$]{ch-co3} for the standard definition by Cheeger-Colding.
This third condition is $iii)$ of page $60$ in \cite{ch-co3} and holds on all limit spaces of Riemannian manifolds we are studying in this paper.

In this paper, we say that a family $\{(C_{i}^l, \phi_{i}^l)\}_{l, i}$ as in Definition \ref{rectifiable} is  a \textit{(weakly) rectifiable coordinate system of $(X, \upsilon)$} if $X$ is (weakly) $\upsilon$-rectifiable. 
See also \cite{keith1} by Keith.
It is important that the cotangent bundle on a rectifiable metric measure space exists in some sense.
We now give several fundamental properties of the cotangent bundle:
\begin{theorem}[Cheeger, Cheeger-Colding, \cite{ch1, ch-co3}]\label{29292}
Let $(X, \upsilon)$ be a rectifiable metric measure space.
Then, there exist a topological space $T^*X$ and a Borel map  $\pi:T^*X \rightarrow X$ with the following properties:
\begin{enumerate}
\item $\upsilon(X \setminus \pi (T^*X))=0$.
\item $\pi^{-1}(w) (=T^*_wX)$ is a finite dimensional real vector space with canonical inner product $\langle \cdot, \cdot \rangle_w$ for every $w \in \pi(T^*X)$ $(|v|(w)= \sqrt{\langle v, v\rangle_w})$.
\item For every Lipschitz function $f$ on $X$, there exist a Borel subset $V$ of $X$, and a Borel map $df$ (called  \textit{the differential of $f$}) from $V$ to $T^*X$ such that $\upsilon(X \setminus V)=0$
and that $\pi \circ df(w)=w$, $|df|(w)=\mathrm{Lip}f(w)=Lipf(w)$ for every $w \in V$, where 
\begin{enumerate}
\item $\mathrm{Lip}f(x)=\lim_{r \to 0}(\sup_{y \in B_r(x)\setminus \{x\}}(|f(x)-f(y)|/\overline{x, y}))$ and
\item $Lip f(x)=\liminf_{r \to 0}(\sup_{y \in \partial B_r(x)}(|f(x)-f(y)|/\overline{x, y})).$
\end{enumerate}
\end{enumerate} 
\end{theorem}
We now give a short review of the construction of the cotangent bundle $T^*X$ as in Theorem \ref{29292}:
Let $\{(C_{i}^l, \phi_{i}^l)\}_{l, i}$ be a rectifiable coordinate system of $(X, \upsilon)$.
By the classical Rademacher's theorem and Definition \ref{rectifiable}, without loss of generality, we can assume that the following properties hold:
\begin{enumerate}
\item Every $\phi_i^l \circ (\phi_j^l)^{-1} : \phi_j^l(C_i^l \cap C_j^l) \to \phi_i^l(C_i^l \cap C_j^l)$ is differentiable at every $w \in \phi_j^l(C_i^l \cap C_j^l)$ (see Section $3.1$ for the notion of differentiability for a Lipschitz function defined on a Borel subset of Euclidean space).
\item For every $i, l$, $x \in C_i^l$ and every $(a_1, \ldots, a_l), (b_1, \ldots, b_l) \in \mathbf{R}^l$, 
\begin{enumerate}
\item $\mathrm{Lip}\left(\sum_ja_j\phi_{i,j}^l\right)(x)=Lip \left(\sum_ja_j\phi^l_{i, j}\right)(x)$,
\item $\mathrm{Lip}\left(\sum_ja_j\phi_{i,j}^l\right)(x)=0$ holds if and only if $(a_1, \ldots, a_l)=0$ holds.
\item $\mathrm{Lip}\left(\sum_j(a_j+b_j)\phi^l_{i,j}\right)(x)^2+\mathrm{Lip}\left(\sum_j(a_j-b_j)\phi^l_{i,j}\right)(x)^2=2\mathrm{Lip}\left(\sum_ja_j\phi^l_{i,j}\right)(x)^2+2\mathrm{Lip}\left(\sum_jb_j\phi^l_{i,j}\right)(x)^2$.
\end{enumerate}
\item For every Lipschitz function $f$ on $X$, we have $\mathrm{Lip}f(x)=Lip f(x)$ for a.e. $x \in X$.
\end{enumerate}
For points $(x, u), (y, v) \in \bigsqcup_{i, l} (\phi_i^l(C_i^l) \times \mathbf{R}^l)$, we define $(x, u) \sim (y, v)$ if $x = \phi_i^l \circ (\phi_j^l)^{-1}(y)$ and $u=J(\phi_i^l \circ (\phi_j^l)^{-1})(y)^tv$ for some $i, j, l$, where $J(f)$ is the Jacobi matrix of a function $f$.
We set $T^*X= \left(\bigsqcup_{i, l}( \phi_i^l(C_i^l) \times \mathbf{R}^l)\right)/ \sim$ and define a map $\pi$ by $\pi (x,u)=(\phi_i^l)^{-1}(x)$ if $x \in \phi_i^l(C_i^l)$.
By the condition $(b)$ above, for every $x \in \pi(T^*X)$ with $x \in C_i^l$, $|a|_x=\mathrm{Lip}\left(\sum_ja_j\phi^l_{i,j}\right)(x)$ is a norm on $\mathbf{R}^l$.
By the condition $(c)$ above, which follows from the $\upsilon$-rectifiable condition $(3)$, we see that the norm comes from an inner product $\langle \cdot, \cdot \rangle_x$ on $\mathbf{R}^l$. 
Then it is easy to check that $(T^*X, \pi, \langle \cdot, \cdot \rangle_x)$ satisfies the conditions as in Theorem \ref{29292}.

See Section $6$ in \cite{ch-co3} by Cheeger-Colding for the details, and page $458-459$ of \cite{ch1} by Cheeger for a more general case.
We set $T^*A=\pi^{-1}(A)$ for every subset $A$ of $X$.
\subsection{Limit spaces of Riemannian manifolds with a lower Ricci curvature bound}
We recall the definition of \textit{Gromov-Hausdorff convergence}.
Let $\{(X_i, x_i)\}_{1 \le i \le \infty}$ be a sequence of pointed proper geodesic spaces.
We say that \textit{$(X_i, x_i)$  Gromov-Hausdorff converges to $(X_{\infty}, x_{\infty})$} if 
there exist sequences of positive numbers $\epsilon_i \to 0$, $R_i \to \infty$ and of maps $\phi_i:  B_{R_i}(x_i) \to B_{R_i}(x_{\infty})$
 (called \textit{an $\epsilon_i$-almost isometry}) with the following three properties: 
\begin{enumerate}
\item $|\overline{x,y}-\overline{\phi_i(x), \phi_i(y)}|<\epsilon_i$ for every $x, y \in B_{R_i}(x_i)$.
\item $B_{R_i}(x_{\infty}) \subset B_{\epsilon_i}(\mathrm{Image} (\phi_i))$.
\item $\phi_i(x_i) \to x_{\infty}$ (denote it by $x_i \to x_{\infty}$ for the sake of simplicity).
\end{enumerate}
See \cite{gr} by Gromov.
We denote it by $(X_i, x_i) \to (X_{\infty}, x_{\infty})$ for brevity.

Moreover, we now give the definition of \textit{measured Gromov-Hausdorff convergence}. 
For a sequence $\{\upsilon_i\}_{1\le i \le \infty}$ of Borel measures $\upsilon_i$ on $X_i$, 
we say that \textit{$\upsilon_{\infty}$ is the limit measure of $\{\upsilon_i\}_i$} if $\upsilon_i(B_r(y_i)) \to \upsilon_{\infty}(B_r(y_{\infty}))$ 
for every $r>0$ and every sequence $\{y_i\}_i$ of points $y_i \in X_i$ with $\phi_i(y_i) \to y_{\infty}$ (denote it by $y_i \to y_{\infty}$).
See \cite{fu} by Fukaya for the original definition and see also \cite{ch-co1} by Cheeger-Colding for this version.
Then we denote it by  $(X_i, x_i, \upsilon_i) \to (X_{\infty}, x_{\infty}, \upsilon_{\infty})$ for brevity.

Let $n \in \mathbf{N}$, $K \in \mathbf{R}$ and let $(M_{\infty}, m_{\infty})$ be a pointed proper geodesic space.
We say that \textit{$(M_{\infty}, m_{\infty})$ is an $(n, K)$-Ricci limit space (of $\{(M_i, m_i)\}_i$)} if there exist
sequences of real numbers $K_i \to K$ and  of pointed complete $n$-dimensional Riemannian manifolds $\{(M_i, m_i)\}_i$ with $\mathrm{Ric}_{M_i} \ge K_i(n-1)$ such that
$(M_i, m_i) \to (M_{\infty}, m_{\infty})$.
We call an $(n, -1)$-Ricci limit space a \textit{Ricci limit space} for brevity.
Moreover we say that a Radon measure $\upsilon$ on $M_{\infty}$ is the \textit{limit measure of $\{(M_i, m_i)\}_i$}
if $\upsilon$ is the limit measure of $\{\mathrm{vol}/\mathrm{vol}\,B_1(m_i)\}_i$.
Then we say that $(M_{\infty}, m_{\infty}, \upsilon)$ is the Ricci limit space of $\{(M_i, m_i, \mathrm{vol}/\mathrm{vol}\,B_1(m_i))\}_i$.

Assume that $(M_{\infty}, m_{\infty}, \upsilon)$ is the Ricci limit space of $\{(M_i, m_i, \mathrm{vol}/\mathrm{vol}\,B_1(m_i))\}_i$.
Cheeger-Colding have proven that the $(1, p)$-Sobolev space $H_{1,p}(U)$ on every open subset $U$ of $M_{\infty}$ is well-defined for every $1< p < \infty$ and that for every $f \in H_{1,p}(U)$, the differential $df(x) \in T^*_xM_{\infty}$ is well-defined for a.e. $x \in U$. See \cite[Theorems $4.14$ and $4.47$]{ch1} by Cheeger for the detail.  

Cheeger-Colding proved the existence of rectifiable coordinate system, defined as in Definition \ref{rectifiable}, constructed from harmonic functions:
\begin{theorem}\cite[Theorem $3.3, 5.5$ and $5.7$]{ch-co3}\label{harm3}
There exists a rectifiable coordinate system $\{(C_{i}^l, \phi_{i}^l)\}_{l, i}$ of $(M_{\infty}, \upsilon)$ such that the following property holds:
There exists a subsequence $\{k(j)\}_j \subset \mathbf{N}$ such that 
for every $l, i$, there exist $x_{\infty} \in M_{\infty}$, $r>0$ with $C_i^l \subset B_r(x_{\infty})$, a sequence $\{x_{k(j)}\}_j$ of $x_{k(j)} \in M_{k(j)}$ with $x_{k(j)} \to x_{\infty}$, 
 a sequence $\{f_{k(j), s}\}_{j, s}$ of harmonic functions $f_{k(j), s}$ on $B_r(x_{k(j)})$
such that $\sup_{j, s}\mathbf{Lip}f_{k(j), s}<\infty$,
$f_{k(j), s} \to \phi_{i, s}^{l}$ on $C_{i}^l$ as $j \to \infty$ for every $s$, where $\mathbf{Lip}f$ is the Lipschitz constant of $f$ and $\phi_{i}^l=(\phi_{i, 1}^l, \ldots, \phi_{i, k}^l)$.
\end{theorem}
See Definition \ref{conv lip} for the definition of the pointwise convergence of functions $f_i \to f_{\infty}$ with respect to the Gromov-Hausdorff topology.
More recently the author proved the existence of rectifiable coordinate system constructed from distance functions:
\begin{theorem}\cite[Theorem $3.1$]{Ho}\label{dist3}
There exists a rectifiable coordinate system $\{(C_{i}^l, \phi_{i}^l)\}_{1 \le l \le n, i<\infty}$ of $(M_{\infty}, \upsilon)$ such that 
every $\phi_{i, s}^l$ is the distance function from a point in $M_{\infty}$.
\end{theorem}
In Section $4$, roughly speaking, we will show the following:
\begin{enumerate}
\item A rectifiable coordinate system as in Theorem \ref{dist3} implies a \textit{weakly $C^{1, \alpha}$-structure of $M_{\infty}$} for some $\alpha=\alpha (n)<1$.
\item A rectifiable coordinate system as in Theorem \ref{harm3} implies a
\textit{weakly second order differential structure of $M_{\infty}$}.
\end{enumerate}
See Corollary \ref{holder} and Theorem \ref{harm} for these precise statements.
\begin{definition}\label{conv lip}
Given functions $f_i: B_R(m_i) \to \mathbf{R}$ and $x_{\infty} \in B_R(m_{\infty})$, we say that
\textit{$f_i$ converges to $f_{\infty}$ at $x_{\infty}$} if for any sequence $x_i \in B_R(m_i)$ such that $x_i \to x_{\infty}$ we have $f_i(x_i) \to f_{\infty}(x_{\infty})$.
We denote this by $f_i \to f_{\infty}$ at $x_{\infty}$.
If this holds for all $x_{\infty} \in B_R(m_{\infty})$ we say $f_i \to f_{\infty}$ on $B_R(m_{\infty})$. 
\end{definition}
Finally, we introduce the definition of a convergence of the differentials of Lipschitz functions with respect to the measured Gromov-Hausdorff topology given in \cite{Ho}.
\begin{definition}\cite[Definitions $1.1$ or $4.4$]{Ho}\label{plplplp}
Given Lipschitz functions $f_i: B_R(m_i) \to \mathbf{R}$ and $x_{\infty} \in B_R(m_{\infty})$,
we say that \textit{$df_i$ converges to $df_{\infty}$ at $x_{\infty}$} if 
\begin{equation}
\sup_i\mathbf{Lip}f_i<\infty
\end{equation}
for every $\epsilon > 0$ and every $z_i \to z_{\infty}$, there exists $r>0$ such that 
\begin{equation}
\limsup_{i \rightarrow \infty}\left|\frac{1}{\mathrm{vol}\,B_t(x_i)}\int_{B_t(x_i)}\langle dr_{z_i}, df_i\rangle d\mathrm{vol}-\frac{1}{\upsilon (B_t(x_{\infty}))}\int_{B_t(x_{\infty})}\langle dr_{z_{\infty}}, df_{\infty}\rangle d\upsilon \right|< \epsilon
\end{equation}
and 
\begin{equation}
\limsup_{i \rightarrow \infty}\frac{1}{\mathrm{vol}\,B_t(x_i)}\int_{B_t(x_i)}|df_i|^2d\mathrm{vol}\le \frac{1}{\upsilon(B_t(x_{\infty}))}\int_{B_t(x_{\infty})}|df_{\infty}|^2d\upsilon + \epsilon
\end{equation}
for every $0 < t < r$ and every $x_i \to x_{\infty}$, where $r_z$ is the distance function from $z$.
We denote this by $df_i \to df_{\infty}$ at $x_{\infty}$.
If this holds for all $x_{\infty} \in B_R(m_{\infty})$ we say $df_i \to df_{\infty}$ on $B_R(m_{\infty})$.
\end{definition}
We write: $(f_i, df_i) \to (f_{\infty}, df_{\infty})$ at $x_{\infty}$ if $f_i \to f_{\infty}$ and $df_i \to df_{\infty}$ at $x_{\infty}$.
\begin{remark}
In \cite{Ho5}, we will see that $df_i$ converges to $df_{\infty}$ on $B_R(m_{\infty})$ in the sense of Definition \ref{plplplp} if and only if $df_i$ $L^p$-converges strongly to $df_{\infty}$ on $B_R(m_{\infty})$ for every $1<p<\infty$.
\end{remark}
We end this subsection by giving three fundamental properties of this convergence which will be used essentially in Section $4$:
\begin{enumerate}
\item If $x_i \to x_{\infty} (x_i \in M_i)$, then $(r_{x_i}, dr_{x_i}) \to (r_{x_{\infty}}, dr_{x_{\infty}})$ on $M_{\infty}$.
\item Let $\{f_i\}_{i \le \infty}$ be a sequence of Lipschitz functions $f_i$ on $B_R(m_i)$ with $\sup_i\mathbf{Lip}f_i<\infty$.
Assume that $f_i$ is a $C^2$-function for every $i < \infty$, $f_i \to f_{\infty}$ on $B_R(m_{\infty})$ and that
\begin{equation}
\sup_{i<\infty}\frac{1}{\mathrm{vol}\,B_R(m_i)}\int_{B_R(m_i)}(\Delta f_i)^2d\mathrm{vol}<\infty.
\end{equation}
Then we have $df_i \to df_{\infty}$ on $B_R(m_{\infty})$.
\item Let $k \in \mathbf{N}$, $\{F_i\}_{1 \le i \le \infty} \subset C^0(\mathbf{R}^k)$ and let $\{f_i^l, g_i^l\}_{1 \le i \le \infty, 1 \le l \le k}$ be a collection of Lipschitz functions $f_i^l, g_i^l$ on $B_R(m_i)$ with $\sup_{l, i} (\mathbf{Lip}\,f_i^l + \mathbf{Lip}\,g_i^l) < \infty$.
Assume that both of the following properties hold:
\begin{enumerate}
\item $F_i$ converges to $F_{\infty}$ with respect to the compact uniform topology on $\mathbf{R}^k$.
\item $df_i^l \rightarrow df_{\infty}^l$ and $dg_i^l \to dg_{\infty}^l$ at a.e. $\alpha \in B_R(m_{\infty})$ for every $1 \le l \le k$.
\end{enumerate}
Then we have 
\begin{equation}
\lim_{i \rightarrow \infty}\frac{1}{\mathrm{vol}\,B_R(m_i)}\int_{B_R(m_{i})}F_i(\langle df_i^1, dg_i^1 \rangle, \ldots, \langle df_i^k, dg_i^k \rangle )d\mathrm{vol}
\end{equation}
\[=\frac{1}{\upsilon (B_R(m_{\infty}))}\int_{B_R(m_{\infty})}F_{\infty}(\langle df_{\infty}^1, dg_{\infty}^1 \rangle, \ldots, \langle df_{\infty}^k, dg_{\infty}^k \rangle )d\upsilon.\]
\end{enumerate}
See \cite[Proposition $4.8$, Corollaries $4.4$ and $4.5$]{Ho} for the proofs.
\section{Weak H$\ddot{\mathrm{o}}$lder continuity and weak Lipschitz continuity}
In this section, we will give several new notions for metric measure spaces and their fundamental properties.
Note that the proofs of these properties are elementary, however, with the theory of convergence of Riemannian manifolds, they perform crucial roles in the analysis of Ricci limit spaces in Section $4$.

We start this section by giving the following definition:
\begin{definition}
Let $A$ be a Borel subset of a metric measure space $(X, \upsilon)$, $Y$ a metric space, $f$ a Borel map from $A$ to $Y$, and $0< \alpha \le 1$.
We say that 
\begin{enumerate}
\item $f$ is \textit{weakly $\alpha$-H$\ddot{o}$lder continuous on $A$} if there exists a countable family $\{A_i\}_i$ of Borel subsets $A_i$ of $A$ such that $\upsilon(A \setminus \bigcup_iA_i)=0$ and that every $f|_{A_i}$ is $\alpha$-H$\ddot{\mathrm{o}}$lder continuous,
\item $f$ is \textit{weakly Lipschitz on $A$} if $f$ is weakly $1$-H$\ddot{\mathrm{o}}$lder continuous on $A$.
\end{enumerate}
\end{definition}
\begin{remark}
Let $M$ be an $n$-dimensional Riemannian manifold, $f$ a function on $M$ and $A$ an open subset of $M$.
Then it is easy to check that the following conditions are equivalent:
\begin{enumerate}
\item $f$ is differentiable at a.e. $x \in A$.
\item $f$ is weakly Lipschitz on $A$.
\end{enumerate}
\end{remark}
\subsection{Weakly twice differentiable functions on a Borel subset of $\mathbf{R}^k$}
Let $A$ be a Borel subset of $\mathbf{R}^k$, $f$ a Lipschitz function on $A$ and $y \in \mathrm{Leb}\,A$, where $\mathrm{Leb}\,A=\{ a \in A; \lim_{r \to 0}H^k(A \cap B_r(a))/H^k(B_r(a))=1\}$.
Then we say that \textit{$f$ is differentiable at $y$} if there exists a Lipschitz function $\hat{f}$ on $\mathbf{R}^k$ such that $\hat{f}|_A=f$ and that $\hat{f}$ is  differentiable at $y$.
Note that if $f$ is differentiable at $y$, then a vector
$(\partial \hat{f}/\partial x_1(y), \ldots, \partial \hat{f}/\partial x_n(y))$
does not depend on the choice of such $\hat{f}$.
Thus we denote the vector by
$J(f)(y)=(\partial f/\partial x_1(y), \ldots, \partial f/\partial x_n(y)).$
Let $F=(f_1, \ldots, f_m)$ be a Lipschitz map from $A$ to $\mathbf{R}^m$.
We say that \textit{$F$ is differentiable at $y$} if every $f_i$ is differentiable at $y$.
Note that by Rademacher's theorem \cite{rad}, $F$ is differentiable at a.e. $x \in A$.
Let us denote the Jacobi matrix of $F$ at $x$ by $J(F)(x)=(\partial f_i/\partial x_j(x))_{ij}$ if $F$ is differentiable at $x$. 
\begin{definition}\label{wert}
Let $\omega=\sum_{i_1< \cdots < i_p} f_{i_1, \ldots, i_p}dx_{i_1} \wedge \cdots \wedge dx_{i_p}$ be a $p$-form on $A$ and $0< \alpha \le 1$.
We say that 
\begin{enumerate}
\item $\omega$ is a \textit{Borel $p$-form on $A$} if every $f_{i_1, \ldots, i_p}$ is a Borel function,
\item $\omega$ is \textit{weakly $\alpha$-H$\ddot{o}$lder continuous on $A$} if every $f_{i_1, \ldots, i_p}$ is weakly $\alpha$-H$\ddot{\mathrm{o}}$lder continuous on $A$,
\item $\omega$ is \textit{weakly Lipschitz on $A$} if every $f_{i_1, \ldots, i_p}$ is weakly Lipschitz on $A$.
\end{enumerate}
\end{definition}
For two Borel $p$-forms $\{\omega_i\}_{i=1,2}$ on $A$, we say that \textit{$\omega_1$ is equivalent to $\omega_2$} if $\omega_1(x)=\omega_2(x)$ for a.e. $x \in A$.
Le us denote the equivalent class of $\omega$ by $[\omega]$, the set of equivalent classes by $\Gamma_{\mathrm{Bor}}(\bigwedge^pT^*A)$, 
the set of equivalent classes represented by a weakly $\alpha$-H$\ddot{\mathrm{o}}$lder continuous $p$-form by $\Gamma_{\alpha}(\bigwedge^pT^*A)$.
We often write $\omega=[\omega]$ for brevity. 

Let $\omega$ be a weakly Lipschitz $p$-form on $A$. Define a Borel $(p+1)$-form $d\omega$ on $A$ by 
$d\omega = \sum_{i_1< \cdots < i_p} (\partial f_{i_1, \ldots, i_p}^j/\partial x_l)dx_l \wedge dx_{i_1} \wedge \cdots \wedge dx_{i_p}$,
where $\omega= \sum_{i_1< \cdots < i_p} f_{i_1, \ldots, i_p}^jdx_{i_1} \wedge \cdots \wedge dx_{i_p}$.
Note that if $\omega_1$ is equivalent to $\omega_2$, then $d\omega_1$ is equivalent to $d\omega_2$.
Therefore $d$ is well-defined as a linear map from $\Gamma_1(\bigwedge^pT^*A)$ to $\Gamma_{\mathrm{Bor}}(\bigwedge^{p+1}T^*A)$.
Note that the following product rule holds:
$d(\eta \wedge \omega) = d\eta \wedge \omega + (-1)^p\eta \wedge d\omega \in \Gamma_{\mathrm{Bor}}(\bigwedge^{p+q}T^*A)$ 
for every $\eta \in \Gamma_1(\bigwedge^pT^*A)$ and every $\omega \in \Gamma_1\left(\bigwedge^qT^*A\right)$.
\begin{lemma}\label{fund}
Let $F$ be a Lipschitz function on $\mathbf{R}^k$.
Then, there exists a sequence $\{F_i\}_i \subset C^{\infty}(\mathbf{R}^k)$ such that 
 $F_i \to F$ in $L^{\infty}(\mathbf{R}^k)$ and that
 $J(F_i)(x) \to J(F)(x)$ for a.e. $x \in \mathbf{R}^k$.
\end{lemma}
\begin{proof}
Let $\rho$ be a nonnegative valued smooth function on $\mathbf{R}^k$ with $\mathrm{supp} (\rho) \subset B_1(0_k)$ and 
\begin{equation}
\int_{\mathbf{R}^k}\rho(x)dH^k=1,
\end{equation}
where $H^k$ is the $k$-dimensional Hausdorff measure.
For every $\epsilon >0$, define smooth functions $\rho_{\epsilon}$ and $F_{\epsilon}$ on $\mathbf{R}^k$ by $\rho_{\epsilon}(x)=\epsilon^{-k} \rho(x/\epsilon)$ and 
\begin{equation}
F_{\epsilon}(x)=\int_{\mathbf{R}^k}\rho_{\epsilon}(x-y)F(y)dH^k.
\end{equation}
Let $L \ge 1$ with $\sup \rho + \mathbf{Lip}F \le L$.
For every $x \in \mathbf{R}^k$, we have 
\begin{align}
\left| F_{\epsilon}(x)-F(x)\right| &\le \int_{\mathbf{R}^k}\rho_{\epsilon}(x-y)|F(y)- F(x)|dH^k\\
&=\int_{B_{\epsilon}(x)}\rho_{\epsilon}(x-y)|F(y)- F(x)|dH^k\\
&\le L \epsilon \int_{B_{\epsilon}(x)}\rho_{\epsilon}(x-y)dH^k\\
&=L \epsilon \int_{B_{\epsilon}(0_k)}\rho_{\epsilon}(y)dH^k\stackrel{\epsilon \to 0}{\to} 0.
\end{align}
Therefore we have the first assertion.
For every $x \in \mathbf{R}^k$ and every $h \in \mathbf{R}$, by the dominated convergence theorem, we have 
\begin{align}
\frac{F_{\epsilon}(x+he_i)-F_{\epsilon}(x)}{h}&=\int_{\mathbf{R}^k}\rho_{\epsilon}(y)\left(\frac{F(x+he_i-y)-F(x-y)}{h}\right)dH^k\\
&=\int_{B_{\epsilon}(0_k)}\rho_{\epsilon}(y)\left(\frac{F(x+he_i-y)-F(x-y)}{h}\right)dH^k\\
&\stackrel{h \to 0}{\to}\int_{B_{\epsilon}(0_k)}\rho_{\epsilon}(y)\frac{\partial F}{\partial x_i}(x-y)dH^k\\
&=\int_{B_{\epsilon}(0_k)}\rho_{\epsilon}(x-y)\frac{\partial F}{\partial x_i}(y)dH^k\\
&=\int_{\mathbf{R}^k}\rho_{\epsilon}(x-y)\frac{\partial F}{\partial x_i}(y)dH^k.
\end{align}
By Lusin's theorem, for every $\delta > 0$ and every $R>0$, there exists a Borel subset $A_{\delta}^R$ of $B_R(0_k)$ such that $H^k(B_R(0_k) \setminus A_{\delta}^R)<\delta$ and that $J(F)|_{A_{\delta}^R}$ is continuous.
Thus, for every $x \in \mathrm{Leb}\,A_{\delta}^R$, we have
\begin{align}
\left| \frac{\partial F_{\epsilon}}{\partial x_i}(x) - \frac{\partial F}{\partial x_i}(x) \right| &\le 
\int_{\mathbf{R}^k}\rho_{\epsilon}(x-y)\left| \frac{\partial F}{\partial x_i}(y) - \frac{\partial F}{\partial x_i}(x) \right| dH^k\\
&=\int_{B_{\epsilon}(x)}\rho_{\epsilon}(x-y)\left| \frac{\partial F}{\partial x_i}(y) - \frac{\partial F}{\partial x_i}(x) \right| dH^k\\
&=\int_{B_{\epsilon}(x) \cap A_{\delta}^R}\rho_{\epsilon}(x-y)\left| \frac{\partial F}{\partial x_i}(y) - \frac{\partial F}{\partial x_i}(x) \right| dH^k\\
& \ \ \ \ +\int_{B_{\epsilon}(x)\setminus A_{\delta}^R}\rho_{\epsilon}(x-y)\left| \frac{\partial F}{\partial x_i}(y) - \frac{\partial F}{\partial x_i}(x) \right| dH^k\\
&\le \sup_{y \in B_{\epsilon}(x) \cap A_{\delta}^R}\left| J(F)(y) - J(F)(x) \right| + 2L \epsilon^{-k}H^k(B_{\epsilon}(x) \setminus A_{\delta}^R)\\
&\stackrel{\epsilon \to 0}{\to} 0.
\end{align}
Since $\delta$ and $R$ are arbitrary, we have the second assertion.
\end{proof}
Let $G=(G_1, \ldots, G_k)$ be a bi-Lipschitz embedding from $A$ to $\mathbf{R}^k$.
For every $\omega  \in \Gamma_{\mathrm{Bor}}(\bigwedge^pT^*G(A))$, define 
$G^*\omega  = \sum f_{i_1, \ldots, i_p} \circ G\, (\partial G_{i_1}/\partial x_{j_1})  \cdots (\partial G_{i_p}/\partial x_{j_p}) dx_{j_1} \wedge \dots \wedge dx_{j_p} \in \Gamma_{\mathrm{Bor}}\left(\bigwedge^pT^*A\right),$
where
$\omega = \sum f_{i_1, \ldots, i_p}dx_{i_1} \wedge \cdots \wedge dx_{i_p}$.
Note that if $J(G)$ is weakly Lipschitz on $A$, then $G^*\omega \in \Gamma_1(\bigwedge^pT^*A)$ for every $\omega \in \Gamma_1(\bigwedge^pT^*G(A))$. 
\begin{proposition}\label{abc}
Let $\omega = \sum f_{i_1, \ldots, i_p}dx_{i_1} \wedge \cdots \wedge dx_{i_p} \in \Gamma_{1}(\bigwedge^pT^*G(A))$.
Assume that $J(G)$ is weakly Lipschitz on $A$.
Then we have $d ( G^*\omega)=G^*(d\omega) \in \Gamma_{\mathrm{Bor}}(\bigwedge^{p+1}T^*A)$.
\end{proposition}
\begin{proof}
Without loss of generality, we can assume that $G$ and every $f_{i_1, \ldots, i_p}$ are Lipschitz on $A$. 
By Lemma \ref{fund}, there exist sequences of smooth maps $\{G^j\}_j$ from $\mathbf{R}^k$ to $\mathbf{R}^k$, and of smooth functions $\{f_{i_1, \ldots, i_p}^j\}_j$ on $\mathbf{R}^k$ such that the following properties hold:
\begin{enumerate}
\item $G^j \to G$ and $f_{i_1, \ldots, i_p}^j \to f_{i_1, \ldots, i_p}$ in $L^{\infty}(A)$.
\item $J(G^j)(x) \to J(G)(x)$ and $J(f_{i_1, \ldots, i_p}^j)(x) \to J(f_{i_1, \ldots, i_p})(x)$ for a.e. $x \in A$.
\end{enumerate}
Since
\begin{equation}
d \left((G^j)^*\left(  \sum f_{i_1, \ldots, i_p}^jdx_{i_1} \wedge \cdots \wedge dx_{i_p} \right)\right)(x)=(G^j)^*\left(d\left( \sum f_{i_1, \ldots, i_p}^jdx_{i_1} \wedge \cdots \wedge dx_{i_p}\right) \right)(x)
\end{equation}
for every $x \in \mathbf{R}^k$ and every $j$, by letting $j \to \infty$, this completes the proof.
\end{proof}
Note that in the same way as Definition \ref{wert}, we can give definitions of Borel vector (tensor) field on $A$, of its equivalence, of its weak $\alpha$-H$\ddot{\mathrm{o}}$lder continuity, and so on.  
Denote the set of equivalent classes of Borel vector fields by $\Gamma_{\mathrm{Bor}}(TA)$ and the set of equivalent classes represented by a weakly Lipschitz vector field by $\Gamma_{1}(TA)$. 

For every weakly Lipschitz function $f$ on $A$ and every $X \in \Gamma_{\mathrm{Bor}}(TA)$, define a Borel function $X(f)=\sum X_i \partial f/ \partial x_i$ on $A$, where $X=\sum X_i\partial / \partial x_i$.
For every $X \in \Gamma_{\mathrm{Bor}}(TA)$, define $G_*X =\sum X(G_i)\partial /\partial x_i \in \Gamma_{\mathrm{Bor}}(TG(A))$.
For every $X, Y \in \Gamma_{1}(TA)$, define 
 $[X, Y] \in \Gamma_{\mathrm{Bor}} (TA)$ by
\begin{equation}
[X, Y]=\sum_{i, j} \left(X_j\frac{\partial Y_i}{\partial x_j}-Y_j\frac{\partial X_i}{\partial x_j}\right)\frac{\partial}{\partial x_j},
\end{equation}
where $X=\sum X_i\partial /\partial x_i$, $Y=\sum Y_i\partial /\partial x_i$.
\begin{proposition}\label{def}
Let $X, Y \in \Gamma_{1}(TA)$.
Assume that $J(G)$ is weakly Lipschitz on $A$.
Then we have $[G_*X, G_*Y]=G_*[X, Y] \in \Gamma_{\mathrm{Bor}}(TG(A))$. 
\end{proposition}
\begin{proof}
The proposition follows from an argument similar to that of the proof of Proposition \ref{abc}.
\end{proof}
\begin{definition}[Weakly twice differentiable function]
Let $f$ be a Borel function on $A$. 
We say that \textit{$f$ is weakly twice differentiable on $A$} if $f$ is weakly Lipschitz function on $A$ and $df \in \Gamma_1(T^*A) (=\Gamma_1(\bigwedge^1T^*A))$.
\end{definition}
Note that the following is \textit{not} trivial.
\begin{proposition}\label{2nd}
Let $f$ be a weakly twice differentiable function on $A$.
Then we have $d(df)=0 \in \Gamma_{\mathrm{Bor}}(\bigwedge^2T^*A)$.
\end{proposition}
\begin{proof}
Let $\hat{A}=\{x \in \mathrm{Leb}\,A; \partial f /\partial x_{1}(x)=\cdots =\partial f /\partial x_{k}(x)=0\}$.
Note that $d(d(f|_{\hat{A}}))=0 \in \Gamma_{\mathrm{Bor}}(\bigwedge^2T^*\hat{A})$ because $d(f|_{\hat{A}})=0 \in \Gamma_{1}(T^*\hat{A})$.
Let $A_i=\{x \in \mathrm{Leb}\,A \setminus \hat{A}; \{df(x)\} \cup \{dx_j(x)\}_{j \neq i}$ is a base of $T_x^*\mathbf{R}^k.\}$.
By \cite[Theorem $3.4$]{Ho}, there exists a countable collection $\{A_{i}^m\}_{1 \le i \le k, m \in \mathbf{N}}$ of Borel subsets $A_{i}^m$ of $A_i$ such that 
$H^k\left((A \setminus \hat{A}) \setminus \bigcup_{i,m}A_i^m\right)=0$
and that every map $\Phi_i^m = (f, x_1, \ldots, x_{i-1}, x_{i+1}, \ldots, x_k)$ is a bi-Lipschitz embedding from $A^m_i$ to $\mathbf{R}^k$.
By the assumption, we see that every $\langle df, dx_j \rangle$ is weakly Lipschitz on $A$.
Therefore,
$J(\Phi_i^m)$, $J((\Phi_i^m)^{-1})$ are weakly Lipschitz on $A_i^m$, $\Phi_i^m(A_i^m)$, respectively.
Since $(\Phi_i^m)^*dx_1 = df$ and $d(dx_1)=0$, the proposition follows directly from Proposition \ref{abc}.
\end{proof}
Let $f$ be a weakly twice differentiable function on $A$.
Put 
\begin{equation}
\frac{\partial ^2 f}{\partial x_i \partial x_j}(x)=\frac{\partial}{\partial x_i}\left(\frac{\partial f}{\partial x_j}\right)(x).
\end{equation}
Note that Proposition \ref{2nd} implies that for every $i, j$ we have
\begin{equation}
\frac{\partial ^2 f}{\partial x_i \partial x_j}(x)=\frac{\partial ^2 f}{\partial x_j \partial x_i}(x)
\end{equation}
for a.e. $x \in A$.

For  $\omega = \sum f_{i_1, \ldots, i_p}dx_{i_1} \wedge \cdots \wedge dx_{i_p} \in \Gamma_1(\bigwedge^pT^*A)$, we say that \textit{$\omega$ is weakly twice differentiable on $A$} if every $f_{i_1, \ldots, i_p}$ is weakly twice differentiable on $A$.
Similarly, we can give definitions of weak twice differentiability for vector (tensor) fields on $A$, for maps from $A$ to $\mathbf{R}^m$, and so on.
\begin{corollary}\label{asdfghjkl}
Let $\omega$ be a weakly twice differentiable $p$-form on $A$.
Then we have $d(d\omega)=0 \in \Gamma_{\mathrm{Bor}}(\bigwedge^{p+2}T^*A)$.
\end{corollary}
\begin{proof}
Since
\begin{equation}
d(d\omega)=\sum d(df_{i_1, \ldots, i_p})\wedge dx_{i_1}\wedge \cdots \wedge dx_{i_p},
\end{equation}
the corollary follows directly from Proposition \ref{2nd}.
\end{proof}
\subsection{Riemannian metric on a Borel subset of $\mathbf{R}^k$}
Let $A$ be a Borel subset of $\mathbf{R}^k$.
In this subsection, we will study \textit{Riemannian metrics on $A$} in the following sense:
\begin{definition}[Riemannian metric]
Let $g=\{g_a\}_{a \in A}$ be a family of inner products $g_a$ on $T_a\mathbf{R}^k$.
We say that 
\begin{enumerate}
\item \textit{$g$ is a Borel Riemannnian metric on $A$} if every $g_{ij}(a)=g_a(\partial/ \partial x_i, \partial/\partial x_j)$ is a Borel function on $A$,   
\item \textit{$g$ is a weakly $\alpha$-H$\ddot{o}$lder continuous Riemannnian metric on $A$} if every $g_{ij}$ is a weakly $\alpha$-H$\ddot{\mathrm{o}}$lder continuous function on $A$,
\item \textit{$g$ is a weakly Lipschitz Riemannnian metric on $A$} if every $g_{ij}$ is a  weakly Lipschitz function on $A$.
\end{enumerate}
\end{definition}
Note that for a Borel Riemannian metric $g$ on $A$, $g$ is weakly $\alpha$-H$\ddot{\mathrm{o}}$lder continuous if and only if $g \in \Gamma_{\alpha}(T^*A \otimes T^*A)$.
For two Borel Riemannian metrics $g, \hat{g}$ on $A$, we say that \textit{$g$ is equivalent to $\hat{g}$} if $g_{ij}$ is equivalent to $\hat{g}_{ij}$. 
Let us denote by $\mathrm{Riem}_{\mathrm{Bor}}(A) (\subset \Gamma_{\mathrm{Bor}}(T^*A \otimes T^*A))$ the set of equivalent classes of Borel Riemannian metrics and  by $\mathrm{Riem}_{\alpha}(A)$ the set of equivalent classes represented by a weakly $\alpha$-H$\ddot{\mathrm{o}}$lder continuous Riemannian metric.

For a weakly Lipschitz function $f$ on $A$, $g \in \mathrm{Riem}_{\mathrm{Bor}}(A)$, $X=\sum X_i \partial /\partial x_i \in \Gamma_{\mathrm{Bor}}(TA)$ and $\omega =\sum \omega_idx_i \in \Gamma_{\mathrm{Bor}}(T^*A)$, define $X^* \in \Gamma_{\mathrm{Bor}}(T^*A)$, $\omega^* \in \Gamma_{\mathrm{Bor}}(TA)$ and $\nabla^gf \in \Gamma_{\mathrm{Bor}}(TA)$ by 
$X^*=\sum g_{ij}X_idx_j, \,\,\omega^*=\sum g^{ij}\omega_i \partial /\partial x_j$
and $\nabla^gf=(df)^*$, respectively, where $g^{ij}$ is the $ij^{th}$ term of the inverse of the matrix defined by $g_{ij}$.
\begin{proposition}[Levi-Civita connection]\label{levi}
Let $g\in \mathrm{Riem}_{1}(A)$.
Then there exists the Levi-Civita connection $\nabla^{g}$ on $A$ defined uniquely in the following sense:
\begin{enumerate}
\item $\nabla^g$ is a map from $\Gamma_{\mathrm{Bor}} (TA) \times \Gamma_{1} (TA)$ to $\Gamma_{\mathrm{Bor}} (TA)$ ($\nabla^g_XY:=\nabla^g(X, Y)$).
\item $\nabla _{X}^{g}(Y+Z)=\nabla _{X}^{g}Y + \nabla _{X}^{g} Z$ for every $X \in \Gamma_{\mathrm{Bor}}(TA)$ and every $Y, Z \in \Gamma_{1}(TA)$.
\item $\nabla _{fX+hY}^{g}Z =f\nabla _{X}^{g}Z + h\nabla _{Y}^{g}Z$ for every $X, Y \in \Gamma_{\mathrm{Bor}}(TA)$, every $Z \in \Gamma_{1}(TA)$ and every Borel functions $f, h$ on $A$.
\item $\nabla _{X}^{g}(fY)=X(f)Y + f\nabla _{X}^{g}Y$ for every $X \in \Gamma_{\mathrm{Bor}}(TA)$, every $Y \in \Gamma_{1}(TA)$ and every weakly Lipschitz function $f$ on $A$.
\item $\nabla _{X}^{g}Y - \nabla _{Y}^{g}X=[X, Y]$ for every $X, Y \in \Gamma_{1}(TA)$.
\item $Xg(Y, Z) = g(\nabla _{X}^{g}Y, Z) + g(Y, \nabla _{X}^{g}Z)$ for every $X \in \Gamma_{\mathrm{Bor}}(TA)$ and every $Y, Z \in \Gamma_{1} (TA)$.   
\end{enumerate}
\end{proposition}
\begin{proof}
Let 
\begin{equation}
\Gamma_{i, j}^m=\frac{1}{2}\sum_l g^{ml}\left( \frac{\partial g_{jl}}{\partial x_i} + \frac{\partial g_{il}}{\partial x_j}-\frac{\partial g_{ij}}{\partial x_l}\right)
\end{equation}
and 
\begin{equation}
\nabla _X^gY=\sum _{i, j}\left( X_i\frac{\partial Y_j}{\partial x_i} \frac{\partial }{\partial x_j} +X_i Y_j\Gamma _{i, j}^m\frac{\partial }{\partial x_m}\right),
\end{equation}
where $X=\sum X_i \partial /\partial x_i$ and $Y=\sum Y_i\partial /\partial x_i$.
It is easy to check that the properties above hold for this $\nabla^g$.
Therefore we have the existence.

Next, we check the uniqueness.
Let $\nabla^1$ and $\nabla ^2$ be Levi-Civita connections on $A$.
Fix $X=\sum X_i\partial /\partial x_i \in \Gamma_{1} (TA), Y=\sum Y_i\partial / \partial x_i \in \Gamma_{1} (TA)$.
Since
\[2g(\nabla ^l _XY, Z)= Xg(Y, Z) + Yg(Z, X)-Zg(X, Y) 
+g([X, Y], Z)-g([Y, Z], X) +g([Z, X], Y) \]
for every $Z \in \Gamma_1 (TA)$ and every $l=1, 2$, we see that $g(\nabla ^1 _XY-\nabla ^2 _XY, Z)=0$ for every $Z \in \Gamma_{1} (TA)$.
Put $\nabla ^1 _XY-\nabla ^2 _XY = \sum h_i \partial / \partial x_i$.
By Lusin's theorem, there exists a sequence of compact subsets $\{A_j\}_j$ of $A$ such that $H^k(A \setminus A_j) \to 0$ as $j \to \infty$ and that $h_i|_{A_j}$ is continuous for every $i, j$. We now make the following elementary claim:
\begin{claim}\label{fund3}
Let $K$ be a bounded Borel subset of $\mathbf{R}^k$, $h$ a continuous function on $K$ and $\epsilon >0$.
Then there exist a Borel subset $K_{\epsilon}$ of $K$ and a Lipschitz function $h_{\epsilon}$ on $\mathbf{R}^k$ such that $|h(x)-h_{\epsilon}(x)|<\epsilon$ for every $x \in K_{\epsilon}$ and that $H^k(K \setminus K_{\epsilon})<\epsilon$.
\end{claim}
The proof is as follows.
For every $x \in \mathrm{Leb}\,K$, there exists $r_x>0$ such that $H^k(B_r(x) \cap K)/H^k(B_r(x)) \ge 1-\epsilon$ for every $0<r<r_x$, and
that $|h(x)-h(y)|<\epsilon$ for every $y \in K \cap B_{r_x}(x)$. 
By standard covering lemma (see for instance Chapter $1$ in \cite{le}), there exists a countable pairwise disjoint collection $\{\overline{B}_{r_i}(x_i)\}_i$ such that 
$x_i \in \mathrm{Leb}\,K$, $r_i<r_{x_i}/5$ and that 
$\mathrm{Leb}\,K \setminus \bigcup_{i=1}^N\overline{B}_{r_i}(x_i) \subset \bigcup_{i=N+1}^{\infty}\overline{B}_{5r_i}(x_i)$
for every $N$.
Fix $N$ with $\sum_{j=N+1}^{\infty}H^k(B_{r_i}(x_i))<\epsilon/5^k$.
Then we have $H^k\left( \mathrm{Leb}\,K \setminus \bigcup_{i=1}^N\overline{B}_{r_i}(x_i)\right)<\epsilon$.
Define a Lipschitz function $f$ on $\bigcup_{i=1}^NB_{r_i}(x_i)$ by $f|_{B_{r_i}(x_i)}\equiv h(x_i)$.
Let $K_{\epsilon}=K \cap \bigcup_{i=1}^NB_{r_i}(x_i)$ and let $h_{\epsilon}$ be a Lipschitz function on $\mathbf{R}^k$ with $h_{\epsilon}|_{K_{\epsilon}}=f$.
Then we have $|h_{\epsilon}(x)-h(x)|<\epsilon$ for every $x \in K_{\epsilon}$.
Thus we have Claim \ref{fund3}.

Therefore there exist collections of Borel subsets $\{A_{j, k}\}_k$ of $A_j$, and of Lipschitz functions $\{h_{i, j, k}\}_{i, j, k}$ on $\mathbf{R}^k$  such that $H^k(A_j \setminus A_{j,k})<2^{-k}$ and that
$|h_i(x)-h_{i, j, k}(x)|<2^{-k}$ for every $x \in A_{j, k}$.
Let $\hat{A}_j=\bigcap_{m=1}^{\infty}\bigcup_{k=m}A_{j,k}$.
Taking $Z=\sum_i h_{i, j, k}\partial / \partial x_i$, we have $g(\sum_i h_i\partial /\partial x_i, \sum_i h_{i, j, k}\partial / \partial x_i)=0$ on $A_{j, k}$.
Letting $k \to \infty$ we have $g(\sum_i h_i\partial /\partial x_i, \sum_i h_{i}\partial / \partial x_i)=0$ and thus $h_i(x)\equiv 0$ on $\hat{A}_j$.
Since $H^k(A_j \setminus \hat{A}_j)=0$, we have $\nabla^1_XY=\nabla^2_XY$.
Therefore for every $\hat{X} =\sum\hat{X}_i\partial/\partial x_i \in \Gamma_{\mathrm{Bor}}(TA)$, we have 
$\nabla_{\hat{X}}^1Y=\sum \hat{X}_i \nabla^1_{\partial/\partial x_i}Y=\sum \hat{X}_i \nabla^2_{\partial/\partial x_i}Y=\nabla^2_{\hat{X}}Y$.
Thus we have the uniqueness.
\end{proof}
Let $G$ be a bi-Lipschitz embedding from $A$ to $\mathbf{R}^k$ and $g \in \mathrm{Riem}_1(A)$.
Assume that $G$ is weakly twice differentiable on $A$, i.e., $J(G)$ is weakly Lipschitz on $A$.
Note that $J(G^{-1})$ is weakly Lipschitz on $G(A)$ and that a Riemannian metric $G_*g$ on $G(A)$ defined by $G_*g(\partial/\partial x_i, \partial /\partial x_j)=g((G^{-1})_*(\partial/\partial x_i), (G^{-1})_*(\partial/\partial x_j))$ is weakly Lipschitz on $G(A)$.
\begin{corollary}\label{levi2}
With the same notation as above, we have $G_*(\nabla _X^gY)=\nabla _{G_*X}^{G_*g}G_*Y$ for every $X \in \Gamma_{\mathrm{Bor}}(TA)$ and every $Y \in \Gamma_{1} (TA)$.
\end{corollary}
\begin{proof}
It is easy to check that if we define a map $T$ from $\Gamma_{\mathrm{Bor}}(TG(A)) \times \Gamma_{1}(TG(A))$ to $\Gamma_{\mathrm{Bor}}(TG(A))$  by  $T(X, Y)=G_*(\nabla^g_{(G^{-1})_*X}(G^{-1})_*Y)$, then $T$ satisfies the properties of the Levi-Civita connection of $G_*g$ on $G(A)$.
Thus the corollary follows from the uniqueness of the Levi-Civita connection.
\end{proof}
\begin{definition}\label{fund4}
Let $f$ be a weakly twice differentiable function on $A$, $\omega \in \Gamma_1(T^*A)$ and $X \in \Gamma_{1}(TA)$.
Define
\begin{enumerate}
\item a Borel tensor field $\nabla^g \omega \in \Gamma_{\mathrm{Bor}}(T^*A \otimes T^*A)$ of type $(0, 2)$ on $A$ by
\begin{equation}
\nabla^g \omega = \sum_{i, j} g\left( \nabla^g_{\frac{\partial}{\partial x_i}}\omega^*, \frac{\partial}{\partial x_j}\right)dx_i \otimes dx_j,
\end{equation}
\item \textit{the Hessian $\mathrm{Hess}^g_f$ of $f$} by $\mathrm{Hess}^g_f=\nabla^g df$,
\item \textit{the divergence $\mathrm{div}^g\,X$ of $X$} by
\begin{equation}
\mathrm{div}^g\,X=\mathrm{trace\,of}\, \nabla^gX^*=\sum_i g\left( \nabla^g_{\frac{\partial}{\partial x_i}}X, \frac{\partial}{\partial x_i}\right),
\end{equation}
\item \textit{the Laplacian $\Delta^g f$ of $f$} by $\Delta^g f= -\mathrm{div}^g\, \nabla^g f$.
\end{enumerate}
\end{definition}
We end this subsection by giving several properties of them:
\begin{corollary}\label{sym}
With the same notation as in Definition \ref{fund4}, we have the following:
\begin{enumerate}
\item $\nabla^g\omega=G^*(\nabla^{G_*g}(G^{-1})^*\omega)$.
\item $\mathrm{Hess}^g_f=G^*(\mathrm{Hess}^{G_*g}_{f \circ G^{-1}})$.
\item $\mathrm{div}^g\,X \circ G^{-1} =\mathrm{div}^{G_*g}\,G_*X$
\item $\mathrm{Hess}^g_f(x)$ is symmetric for a.e. $x \in A$.
\item $\mathrm{div}^g\,(h(\nabla^g f))=-h\Delta^g f +g(\nabla^g f, \nabla^g h)$ for every weakly twice differentiable function $h$ on $A$. 
\item $\Delta^g (fh)=h\Delta^g f-2g(\nabla^g f, \nabla^g h)+f\Delta^g h$ for every weakly twice differentiable function $h$ on $A$.
\end{enumerate}
\end{corollary}
\begin{proof}
$(1), (2)$ and $(3)$ all follow directly from Corollary \ref{levi2}.
Since
\begin{equation}
\mathrm{Hess}^g_f=\sum_{i, j} \left(\frac{\partial ^2 f}{\partial x_i \partial x_j}-\left(\nabla^g_{\frac{\partial}{\partial x_i}}\frac{\partial}{\partial x_j}\right)(f)\right) dx_i \otimes dx_j,
\end{equation}
$(4)$ follows from Proposition \ref{2nd}.
On the other hand, by simple calculations, we have $(5)$ and $(6)$.
\end{proof}
\subsection{Weakly second order differential structure on weakly rectifiable metric measure spaces}
In this subsection, we will discuss some weak twice differentiability on weakly rectifiable metric measure spaces.
\begin{definition}[Weakly second order differential structure]\label{2nd2nd2nd}
Let $(X, \upsilon)$ be a metric measure space and $0<\alpha \le 1$. We say that 
\begin{enumerate}
\item \textit{$(X, \upsilon)$ has a weakly $C^{1, \alpha}$-structure} if there exists a weakly rectifiable coordinate system $\{(C_i^l, \phi_i^l)\}_{l, i}$ of $(X, \upsilon)$ such that every Jacobi matrix map $J(\Phi_{ij}^l)$ of $\Phi_{ij}^l=\phi_j^l \circ (\phi_i^l)^{-1}$ from $\phi_i^l(C_i^l \cap C_j^l)$ to $\phi_j^l(C_i^l \cap C_j^l)$ is weakly $\alpha$-H$\ddot{\mathrm{o}}$lder continuous, 
\item \textit{$(X, \upsilon)$ has a weakly second order differential structure} if $(X, \upsilon)$ has a weakly $C^{1, 1}$-structure.
\end{enumerate}
\end{definition}
\begin{definition}\label{pppo}
Let $0<\hat{\alpha}\le \alpha \le 1$, let $(X, \upsilon)$ be a metric measure space having a weakly $C^{1, \alpha}$-structure with respect to $\{(C_i^l, \phi_i^l)\}_{l, i}$, a Borel subset $A$ of $X$ and
$\omega=\{\omega_i^l\}_{l, i}$ a family of Borel $p$-forms $\omega_i^l$ on $\phi_i^l(C_i^l \cap A)$.
We say that 
\begin{enumerate}
\item $\omega$ is a \textit{Borel $p$-form on $A$} if  
$(\Phi_{ij}^l)^* \omega_j^l=\omega_i^l$ on $\phi_i^l(C_i^l \cap C_j^l \cap A)$ for every $i, l, j$ with $\upsilon(C_i^l \cap C_j^l \cap A)>0$.
\item $\omega$ is a \textit{weakly $\hat{\alpha}$-H$\ddot{o}$lder continuous $p$-form on $A$} if $\omega$ is a Borel $p$-form on $A$, and $\omega_i^l \in \Gamma_{\hat{\alpha}}(\bigwedge^pT^*\phi_i^l(C_i^l \cap A))$, 
\end{enumerate}
Write $\omega |_{C_i^l \cap A}=\omega_i^l$.  
\end{definition}
Note that $\omega$ can be identified as a Borel section from $A$ to a $L^{\infty}$-vector bundle $\bigwedge^pT^*X$ on $X$.
See Section $4$ in \cite{ch} by Cheeger or Section $6$ in \cite{ch-co3} by Cheeger-Colding for the details. 
Note that in the same way as in Definition \ref{pppo}, we can give definitions of Borel vector (tensor) field on $A$, its H$\ddot{\mathrm{o}}$lder continuity, its equivalence, and so on.
For instance, denote the set of equivalent classes of Borel sections $s: A \to T^*X \otimes T^*X$ by $\Gamma_{\mathrm{Bor}}(T^*A \otimes T^*A)$.
Similarly, define $\Gamma_{\mathrm{Bor}}(T^*A)$, $\Gamma_{\mathrm{Bor}}(\bigwedge^pT^*A)$, and so on.

Assume that $(X, \upsilon)$ has a weakly second order differential structure with respect to  $\{(C_i^l, \phi_i^l)\}_{l, i}$.
Let $A$ be a Borel subset of $X$.
As in the previous section, denote the set of equivalent classes of Borel vector fields on $A$ represented by a weakly Lipschitz vector field on $A$ by $\Gamma_{1}(TA, \{(C_i^l, \phi_i^l)\}_{l, i})$.
Similarly define $\Gamma_{1}(T^*A \otimes T^*A, \{(C_i^l, \phi_i^l)\}_{l, i})$ and so on.
We often write: $\Gamma_{1}(\bigwedge^pT^*A)=\Gamma_{1}(\bigwedge^pT^*A, \{(C_i^l, \phi_i^l)\}_{l, i})$, $\Gamma_1(T^*A \otimes T^*A)=\Gamma_{1}(T^*A \otimes T^*A, \{(C_i^l, \phi_i^l)\}_{l, i})$, and so on, for brevity.
\begin{proposition}
Let $\omega \in \Gamma_1(\bigwedge^pT^*A)$.
Then there exists $d\omega \in \Gamma_{\mathrm{Bor}}(\bigwedge^{p+1}T^*A)$ defined uniquely such that $d\omega|_{C_i^l\cap A}=d(\omega|_{\phi_i^l(C_i^l \cap A)})$.
\end{proposition}
\begin{proof}
This is a direct consequence of Proposition \ref{abc}.
\end{proof}
\begin{proposition}
Let $V, W \in \Gamma_{1}(TA)$.
Then there exists  $[V, W] \in \Gamma_{\mathrm{Bor}}(TA)$ defined uniquely such that $[V, W]|_{C_i^l\cap A}=[V_{C_i^l \cap A}, W_{C_i^l \cap A}]$.
\end{proposition}
\begin{proof}This is a direct consequence of Proposition \ref{def}.
\end{proof}
\begin{definition}[Weakly twice differentiable function]\label{2ndd}
We say that a Borel function $f$ on $A$ is \textit{weakly twice differentiable on $A$ with respect to $\{(C_i^l, \phi_i^l)\}_{l, i}$} if every $f \circ (\phi^l_i)^{-1}$ is weakly twice differentiable on $\phi_i^l(C_i^l \cap A)$.
\end{definition}
The following is a direct consequence of Proposition \ref{2nd}:
\begin{corollary}
Let $f$ be a weakly twice differentiable function on $A$ with respect to $\{(C_i^l, \phi_i^l)\}_{l, i}$.
Then we have $d(df)=0 \in \Gamma_{\mathrm{Bor}}(\bigwedge^2T^*A)$.
\end{corollary}
We say that $g  \in \Gamma_{\mathrm{Bor}}(T^*A \otimes T^*A)$ is a \textit{Borel Riemannian metric on $A$} if $g$ is symmetric and positive definite.
For a Borel Riemannian metric $g$ on $A$, a weakly Lipschitz function $f$ on $A$, $X \in \Gamma_{\mathrm{Bor}}(TA)$ and $\omega \in \Gamma_{\mathrm{Bor}}(T^*A)$, in the same way as in the previous subsection, define $X^* \in \Gamma_{\mathrm{Bor}}(T^*A)$, $\omega^* \in \Gamma_{\mathrm{Bor}}(TA)$ and $\nabla^gf \in \Gamma_{\mathrm{Bor}}(TA)$.
\begin{remark}
We can \textit{not} discuss a twice differentiability for vector field (or $p (\ge 1)$-form) on $X$ in the same way as above. 
\end{remark}
\subsection{Weakly Lipschitz Riemannian metric on weakly rectifiable metric measure spaces}
In this subsection, we will study Riemannian metrics on weakly rectifiable metric measure spaces.
Let $(X, \upsilon)$ be a weakly rectifiable metric measure space with respect to $\{(C_i^l, \phi_i^l)\}_{l, i}$, and 
$g =\{g_i^l\}_{l, i} \in \Gamma_{\mathrm{Bor}}(T^*X \otimes T^*X)$ a Borel Riemannian metric on $X$. 
\begin{definition}[Weakly Lipschitz Riemannian metric on weakly rectifiable metric measure spaces]\label{wlr}
We say that 
\begin{enumerate}
\item $g$ is a \textit{weakly $\alpha$-H$\ddot{o}$lder Riemannian metric on $X$ with respect to $\{(C_i^l, \phi_i^l)\}_{l, i}$} if $g_i^l \in \mathrm{Riem}_{\alpha}(\phi_i^l(C_i^l))$,
\item $g$ is a \textit{weakly Lipschitz Riemannian metric on $X$ with respect to $\{(C_i^l, \phi_i^l)\}_{l, i}$} if $g_i^l \in \mathrm{Riem}_{1}(\phi_i^l(C_i^l))$.
\end{enumerate}
\end{definition}
\begin{proposition}\label{we2nd}
Let $0<\alpha \le 1$.
Assume that $g$ is a weakly $\alpha$-H$\ddot{o}$lder Riemannian metric on $X$ with respect to $\{(C_i^l, \phi_i^l)\}_{l, i}$.
Then 
we see that $(X, \upsilon)$ has a weakly $C^{1, \alpha}$-structure with respect to $\{(C_i^l, \phi_i^l)\}_{l, i}$ and that 
$g \in \Gamma_{\alpha}(T^*X \otimes T^*X, \{(C_i^l, \phi_i^l)\}_{l, i})$.
\end{proposition}
\begin{proof}Since $g_i^l$ is weakly $\alpha$-H$\ddot{\mathrm{o}}$lder continuous, by simple calculation, we see that every map $J(\Phi_{ij}^l)$ is weakly $\alpha$-H$\ddot{\mathrm{o}}$lder continuous on 
$\phi_i^l(C_i^l \cap C_j^l)$.
Therefore we have the proposition.
\end{proof}
Assume that $g$ is a weakly Lipschitz Riemannian metric on $X$ with respect to $\{(C_i^l, \phi_i^l)\}_{l, i}$.
The following is a direct consequence of Proposition \ref{levi} and Corollary \ref{levi2}:
\begin{theorem}[Levi-Civita connection]\label{levi3}
There exists the Levi-Civita connection $\nabla^g$ on $X$ defined uniquely in the following sense:  
\begin{enumerate}
\item $\nabla^g$ is a map from $\Gamma_{\mathrm{Bor}} (TX) \times \Gamma_{1} (TX)$ to $\Gamma_{\mathrm{Bor}} (TX)$.
\item $\nabla^g_U(V+W)=\nabla^g_UV + \nabla^g_U W$ for every $U \in \Gamma_{\mathrm{Bor}} (TX)$ and every $V, W \in \Gamma_{1} (TX)$.
\item $\nabla^g_{fU+hV}W =f\nabla^g_UW + h\nabla^g_VW$ for every $U, V \in \Gamma_{\mathrm{Bor}} (TX)$, every $W \in \Gamma_{1} (TX)$ and every Borel functions $f, h$ on $X$.
\item $\nabla^g_U(fV)=U(f)V + f\nabla^g_UV$ for every $U \in \Gamma_{\mathrm{Bor}} (TX)$, every $V \in \Gamma_{1} (TX)$ and every weakly Lipschitz function $f$ on $X$.
\item $\nabla^g_UV - \nabla^g_VU=[U, V]$ for every $U, V \in \Gamma_{1} (TX)$.
\item $Ug(V, W) = g( \nabla^g_UV, W) + g(V, \nabla^g_UW)$ for every $U \in \Gamma_{\mathrm{Bor}} (TX)$ and every $V, W \in \Gamma_{1} (TX)$.   
\end{enumerate}
\end{theorem}
The following is a direct consequence of Corollary \ref{sym}:
\begin{proposition}\label{hess2}
Let $A$ be a Borel subset of $X$, $f$ a weakly twice differentiable function on $A$, $\omega \in \Gamma_1(T^*A)$ and $Y \in \Gamma_{1}(TA)$.
Then there exist uniquely
\begin{enumerate}
\item $\nabla^g \omega \in \Gamma_{\mathrm{Bor}}(T^*A \otimes T^*A)$ satisfying that $\nabla^g \omega|_{C_i^l \cap A}=\nabla^{g_i^l}(\omega|_{C_i^l\cap A})$,
\item the Hessian $\mathrm{Hess}^g_f \in \Gamma_{\mathrm{Bor}}(T^*A \otimes T^*A)$ satisfying that $\mathrm{Hess}^g_f|_{C_i^l \cap A}=\mathrm{Hess}_{f \circ \phi_i^l}^{g_i^l}$,
\item a Borel function  $\mathrm{div}^g\,Y$ (called the divergence of $Y$) on $A$ satisfying that $\mathrm{div}^g\,Y(x) = \mathrm{div}^{g_i^l}\,(Y|_{C_i^l})(\phi_i^l(x))$
for a.e. $x \in \phi_i^l(A \cap C_i^l)$,
\item a Borel function $\Delta^g f$ on $A$ satisfying that  $\Delta^g f(x) = \Delta^{g_i^l} (f \circ (\phi_i^l)^{-1})(\phi_i^l(x))$
for a.e. $x \in \phi_i^l(A \cap C_i^l)$.
\end{enumerate}
Moreover we have the following:
\begin{itemize}
\item[(a)] $\mathrm{Hess}^g_f(x)$ is symmetric for a.e. $x \in A$.
\item[(b)] $\mathrm{div}^g\,(h(\nabla^g f))=-h\Delta^g f +g(\nabla^g f, \nabla^g h)$ for every weakly twice differentiable function $h$ on $A$. 
\item[(c)] $\Delta^g (fh)=h\Delta^g f-2g(\nabla^g f, \nabla^g h)+f\Delta^g h$ for every weakly twice differentiable function $h$ on $A$.
\end{itemize}
\end{proposition}
Finally we end this subsection by giving the definition of the canonical Riemannian metric on a rectifiable metric measure space:
\begin{definition}
Let $(\hat{X}, \hat{\upsilon})$ be a rectifiable metric measure space with respect to $\{(\hat{C}_i^l, \hat{\phi}^l_i)\}_{l,i}$, and
 $\{\langle \cdot, \cdot \rangle_w\}_w$ the canonical family of inner products $\langle \cdot, \cdot \rangle_w$ on $T_w^*\hat{X}$ as in Theorem \ref{29292}.
Define $\hat{g}=\{\hat{g}_i^l\}_{l, i} \in \Gamma_{\mathrm{Bor}}(T^*X \otimes T^*X)$ by $(\hat{g}_i^l)^{st}=\langle d\hat{\phi}_{i, s}^l, d\hat{\phi}_{i, t}^l\rangle$, where
$\hat{\phi}_i^l=(\hat{\phi}_{i,1}^l, \ldots, \hat{\phi}_{i, k}^l)$,
and call $\hat{g}=\{\hat{g}_i^l\}_{l, i} \in \Gamma_{\mathrm{Bor}}(T^*\hat{X} \otimes T^*\hat{X})$ the \textit{Riemannian metric of $(\hat{X}, \hat{\upsilon})$ with respect to $\{(\hat{C}_i^l, \hat{\phi}^l_i)\}_{l,i}$}.
\end{definition}
\begin{remark}
In \cite{ch1, ch-co3} by Cheeger and Cheeger-Colding refer to the family $\{\langle \cdot, \cdot \rangle_w\}_w$ as the Riemannian metric of $(\hat{X}, \hat{\upsilon})$.
\end{remark}
\section{Ricci limit spaces}
In this section, we will give the proofs of Theorems \ref{angle} and \ref{limit}.
Let $(M_{\infty}, m_{\infty})$ be a Ricci limit space.
Cheeger-Colding showed that any two limit measures $\upsilon_1, \upsilon_2$ on $M_{\infty}$ are mutually absolutely continuous.
See \cite[Theorem $4.17$]{ch-co3}.
Therefore, for instance, note that the notion of weak H$\ddot{\mathrm{o}}$lder continuity for functions on $M_{\infty}$ does not depend on the choice of the limit measures. 
\subsection{Angles, its weak H$\ddot{\mathbf{o}}$lder continuity, and bi-Lipschitz embedding}
In this subsection, we will give a proof of Theorem \ref{angle} and discuss the weak H$\ddot{\mathrm{o}}$lder continuity of angles.
The proof of the following proposition is based on the proof of Cheeger-Colding's splitting theorem \cite[Theorem $6.64$]{ch-co}:
\begin{proposition}\label{hol}
For every $\epsilon>0$, there exists $\delta=\delta(\epsilon, n)>0$ such that the following property holds:
Let $M$ be an $n$-dimensional complete Riemannian manifold with $\mathrm{Ric}_M \ge -\epsilon ^2(n-1)$, and $m, p_1, p_2, q_1, q_2 \in M$.
Assume that $\overline{p_i, m} \ge \epsilon ^{-1}, \overline{q_i, m} \ge \epsilon ^{-1}$ and $\overline{m, p_i} + \overline{m, q_i}- \overline{p_i, q_i} \le \delta$ for $i=1,2$.
Then we have 
\begin{equation}
\frac{1}{\mathrm{vol}\,B_1(m)}\int_{B_1(m)}\left| \langle dr_{p_1}, dr_{p_2} \rangle -\frac{1}{\mathrm{vol}\,B_1(m)}\int_{B_1(m)}\langle dr_{p_1}, dr_{p_2} \rangle d\mathrm{vol}\right| d\mathrm{vol} \le C(n) \epsilon^{\alpha (n)},
\end{equation}
where $0 < \alpha (n) <1$ and $C(n) \ge 1$ are constants depending only on $n$.
\end{proposition}
\begin{proof}
Cheeger-Colding's proof of \cite[Lemmas $6.15$, $6.25$ and Proposition $6.60$]{ch-co} yields that there exists $\delta=\delta(\epsilon, n)>0$ with the following properties:
Let $M, m,  p_1, p_2, q_1, q_2$ be as above.
Then for every $i=1,2$, there exists a harmonic function $\mathbf{b}_{i}$ on $B_{1}(m)$ such that $|r_{p_i}-\mathbf{b}_{i}|_{L^{\infty}(B_{1}(m))} \le C_1(n) \epsilon^{\alpha_1 (n)}$ and
\begin{equation}
\frac{1}{\mathrm{vol}\,B_1(m)}\int_{B_1(m)}\left(|dr_{p_i}-d\mathbf{b}_i|^2 + |\mathrm{Hess}_{\mathbf{b}_i}|^2\right)d\mathrm{vol}\le C_2(n)\epsilon ^{\alpha_2(n)}.
\end{equation}
Therefore, by the Poincar\'e inequality of type $(1,2)$ on $M$, we have 
\begin{align}
&\frac{1}{\mathrm{vol}\,B_1(m)}\int_{B_1(m)}\left| \langle d\mathbf{b}_1, d\mathbf{b}_2 \rangle -\frac{1}{\mathrm{vol}\,B_1(m)}\int_{B_1(m)}\langle d\mathbf{b}_1. d\mathbf{b}_2 \rangle d\mathrm{vol}\right| d\mathrm{vol} \\
&\le C(n)
\sqrt{\frac{1}{\mathrm{vol}\,B_1(m)}\int_{B_1(m)}\left(|\mathrm{Hess}_{\mathbf{b}_1}|^2+|\mathrm{Hess}_{\mathbf{b}_2}|^2\right)d\mathrm{vol}} \le C_3(n) \epsilon^{\alpha _3(n)}.
\end{align}
This completes the proof.
\end{proof}
\begin{remark}
These H$\ddot{\mathrm{o}}$lder estimates as above essentially follows from Laplacian comparison theorem and the following Abresch-Gromoll excess estimate: 
\begin{equation}
e(x) \le C(n)\epsilon^{\alpha(n)}
\end{equation}
for $M, m, p_1, p_2, q_1, q_2$ as in Proposition \ref{hol} and for every $x \in B_{100}(m)$, where $e(x)=\overline{p_1, x}+\overline{q_1, x} -\overline{p_1, q_2}$.
See \cite[Proposition $6.2$]{ch-co} by Cheeger-Colding (or \cite[Theorem $9.1$]{ch} by Cheeger) for the details.
We can also see a valuable survey about this excess estimate in subsection $1.7$ in \cite{co-na1} and \cite[Theorem $2.6$]{co-na1} by Colding-Naber for a new very useful excess estimate. 
\end{remark}
For every $p \in M_{\infty}$ and every $\tau>0$, put $\mathcal{D}^{\tau}_p=\{z \in M_{\infty};$ There exists $w \in M_{\infty}$ with $\overline{w, z}\ge \tau$ and $\overline{p, z}+\overline{z,w}=\overline{p,w}.\}$.
Note $C_p = M_{\infty} \setminus \bigcup_{\tau>0}\mathcal{D}^{\tau}_p$.
\begin{corollary}\label{hol1}
Let $\upsilon$ be a limit measure on $M_{\infty}$, $\beta>0$, $\tau >0$, $p, q \in M_{\infty}$ and $x \in \mathcal{D}_p^{\tau} \cap \mathcal{D}_q^{\tau} \setminus (B_{\beta}(p) \cup B_{\beta}(q))$.
Then we have 
\begin{equation}
\frac{1}{\upsilon (B_r(x))}\int_{B_r(x)}\left| \langle dr_p, dr_q \rangle -\frac{1}{\upsilon (B_r(x))}\int_{B_r(x)}\langle dr_p, dr_q \rangle d\upsilon \right| d\upsilon \le C(n) \max \left\{r, \frac{r}{\beta}, \frac{r}{\tau}\right\}^{\alpha (n)}
\end{equation}
for every $0 < r < \min \{\beta, \tau\}$.
\end{corollary}
\begin{proof}
Let $\{(M_i, m_i)\}_i$ be a sequence of pointed $n$-dimensional complete Riemannian manifolds with $\mathrm{Ric}_{M_i} \ge -(n-1)$ such that $(M_i, m_i, \mathrm{vol}/\mathrm{vol}\,B_1(m_i)) \to (M_{\infty}, m_{\infty}, \upsilon)$. 
Fix sequences $p_i, q_i, x_i \in M_i$ with $p_i \to p, q_i \to q$ and $x_i \to x$.
By considering the rescaled metric $r^{-1}d_{M_i}$, it follows from Proposition \ref{hol} that
\begin{equation}
\frac{1}{\mathrm{vol}\,B_r(x_i)}\int_{B_r(x_i)}\left| \langle dr_{p_i}, dr_{q_i} \rangle -\frac{1}{\mathrm{vol}\,B_r(x_i)}\int_{B_r(x_i)}\langle dr_{p_i}. dr_{q_i} \rangle d\mathrm{vol}\right| d\mathrm{vol} \le C(n) \max \left\{r, \frac{r}{\beta}, \frac{r}{\tau}\right\}^{\alpha (n)}
\end{equation}
for every sufficiently large $i$.
By the property $(1)$ in page $8$, since $dr_{p_i} \to dr_p, dr_{q_i} \to dr_q$ and $dr_{x_i} \to dr_x$ on $M_{\infty}$, by letting $i \to \infty$ and the property $(3)$ in page $8$, we have the corollary.
\end{proof}
Let $\upsilon$ be a limit measure on $M_{\infty}$ and $p, q, x \in M_{\infty}$ with $x \in M_{\infty} \setminus (C_p \cup C_q \cup \{p\} \cup \{q\})$.
It follows directly from Corollary \ref{hol1} that the limit
\begin{equation}
\lim_{r \to 0} \frac{1}{\upsilon ( B_r(x))}\int_{B_r(x)}\langle dr_p, dr_q\rangle d\upsilon
\end{equation}
exists.
Define the angle $\angle^{\upsilon} pxq$ of $pxq$ with respect to $\upsilon$ by
\begin{equation}
\angle^{\upsilon} pxq = \arccos \left(\lim_{r \to 0} \frac{1}{\upsilon ( B_r(x))}\int_{B_r(x)}\langle dr_p, dr_q\rangle d\upsilon \right).
\end{equation}
\begin{theorem}\label{well3}
We have 
\begin{equation}
\cos \angle^{\upsilon}pxq= \lim_{t \to 0}\frac{2t^2-\overline{\gamma_p(t), \gamma_q(t)}^2}{2t^2}
\end{equation}
for any minimal geodesics $\gamma_p$ from $x$ to $p$, and $\gamma_q$ from $x$ to $q$.
In particular, we have Theorem \ref{angle}.
\end{theorem}
\begin{proof}
Recall that a pointed proper geodesic space $(Y, y)$ is said to be a \textit{tangent cone of $M_{\infty}$ at $x \in M_{\infty}$} if there exists a sequence of positive numbers $\{r_i\}_i$ such that $r_i \to 0$ and $(M_{\infty}, r_i^{-1}d_{M_{\infty}}, x) \to (Y, y)$. 
Fix 
\begin{enumerate}
\item a sequence of positive numbers $\{r_i\}_i$ with $r_i \to 0$,
\item a tangent cone $(Y, y)$ of $M_{\infty}$ at $x$, and a Radon measure $\upsilon_Y$ on $Y$  
 satisfying that $(M_{\infty}, r^{-1}_id_{M_{\infty}}, x, \upsilon_i) \to (Y, y, \upsilon_Y)$, where $\upsilon_i=\upsilon / \upsilon (B_{r_i}(x))$,
\item two geodesics $\gamma_p, \gamma_q$ on $M_{\infty}$ beginning at $p$, $q$, respectively, such that $x$ is an interior point of both $\gamma_p$ and $\gamma_q$.
\end{enumerate}
Then it is easy to check that there exist lines $l_p, l_q$ of $Y$ such that $y \in \mathrm{Image}(l_p) \cap \mathrm{Image}(l_q)$ and that $(\gamma_p, r_i^{-1}d_Y) \to l_p$  and $(\gamma_q, r_i^{-1}d_Y) \to l_q$ with respect to the Gromov-Hausdorff topology (recall that a map $l: \mathbf{R} \to Y$ is said to be a \textit{line of $Y$} if $l$ is an isometric embedding).
\begin{claim}\label{bus}
Let $b^i_p=r_i^{-1}d_{M_{\infty}}(p, \cdot)-r_i^{-1}d_{M_{\infty}}(p, x)$.
Then $(b^i_p, db^i_p) \to (b_{l_p}, db_{l_p})$ on $Y$ with respect to the convergence $(M_{\infty}, r^{-1}_id_{M_{\infty}}, x, \upsilon_i) \to (Y, y, \upsilon_Y)$, where $b_{l_p}$ is the Busemann function of $l_p$.
\end{claim}
The proof is as follows.
It is easy to check $b_p^i \to b_{l_p}$ on $Y$.
Let $R>0$, $x_j, p_j \in M_j$ with $x_j \to x, p_j \to p$ and let $b^{i, j}=r_i^{-1}d_{M_{j}}(p_j, \cdot)-r_i^{-1}d_{M_{j}}(p_j, x_j)$.
\cite[Lemmas $6.15$ and $6.25$]{ch-co} by Cheeger-Colding yields that there exists 
a sequence $\{\mathbf{b}^{i, j}\}_{i<\infty, j < \infty}$ of $C(n)$-Lipschitz harmonic functions $\mathbf{b}^{i, j}$ on $B_R^{r_i^{-1}d_{M_j}}(x_{j})$ such  that for every $i$ there exists $i_0$ such that $||\mathbf{b}^{i, j}- b^{i, j}||_{L^{\infty}(B_R^{r_i^{-1}d_{M_j}}(x_j))} +||d\mathbf{b}^{i, j}- db^{i, j}||_{L^{2}(B_R^{r_i^{-1}d_{M_j}}(x_j))} \le \Psi(r_i^{-1}; n, R)$
for every $j \ge i_0$.
Without loss of generality we can assume that for every $i$ there exists a $C(n)$-Lipschitz function $\mathbf{b}^{i, \infty}$ on $B_R^{r_i^{-1}d_{M_{\infty}}}(x)$ such that $\mathbf{b}^{i, j} \to \mathbf{b}^{i, \infty}$ on $B_R^{r_i^{-1}d_{M_{\infty}}}(x)$.
Note that the property $(2)$ in page $8$ yields $d\mathbf{b}^{i, j} \to d\mathbf{b}^{i, \infty}$ on $B_R^{r_i^{-1}d_{M_{\infty}}}(x)$.
In particular we have $||\mathbf{b}^{i, \infty}- b^{i}_p||_{L^{\infty}(B_R^{r_i^{-1}d_{M_{\infty}}}(x))} +||d\mathbf{b}^{i, \infty}- db^{i}_p||_{L^{2}(B_R^{r_i^{-1}d_{M_{\infty}}}(x))} \le \Psi(r_i^{-1}; n, R)$.
Thus $\mathbf{b}^{i, \infty} \to b_{l_p}$ on $B_R(y)$.
Since there exists a subsequence $\{j(i)\}_i$ such that $\mathbf{b}^{i, j(i)} \to b_{l_p}$ on $B_R(y)$ with respect to the convergence $(M_{j(i)}, x_{j(i)}, r_{i}^{-1}d_{M_{j(i)}}) \to (Y, y)$, 
applying the property $(2)$ in page $8$ again yields $d\mathbf{b}^{i, \infty} \to db_{l_p}$ on $B_R(y)$.
Thus we have $db^i_p \to db_{l_p}$ on $B_R(y)$.
Since $R$ is arbitrary, we have Claim \ref{bus}.

Therefore we have 
\begin{align}
\cos \angle^{\upsilon} pxq&= \lim_{i \to \infty} \frac{1}{\upsilon(B_{r_i}(x))}\int_{B_{r_i}(x)}\langle dr_{p}, dr_q\rangle d\upsilon \\
&=\lim_{i \to \infty} \frac{1}{\upsilon_i(B_{1}^{r_i^{-1}d_{M_{\infty}}}(x))}\int_{B_{1}^{r_i^{-1}d_{M_{\infty}}}(x)}\langle db^i_{p}, db^i_{q}\rangle d\upsilon_i \\
&=\frac{1}{\upsilon_Y(B_1(y))}\int_{B_1(y)}\langle db_{l_p}, db_{l_q}\rangle d\upsilon_Y \\
&=\cos (\mathrm{the \,angle\,between\,}l_p\,\mathrm{and}\,l_q)\\
&=\lim_{i \to \infty}\frac{2r_i^2-\overline{\gamma_p(r_i), \gamma_q(r_i)}^2}{2r_i^2}.
\end{align}
Thus Gromov's compactness theorem yields the theorem.
\end{proof}
Since $\angle^{\upsilon}pxq$ is independent of $\upsilon$, we set $\angle pxq=\angle^{\upsilon}pxq$ and call it \textit{the angle of $pxq$}.  
\begin{remark}
By the proof of Theorem \ref{well3} and the property of $(3)$ in page $8$, we have 
\begin{equation}
\lim_{\epsilon \to 0}\sup_{s, t \in [\tau_1, \tau_2]}\left|\cos \angle pxq - \frac{(s\epsilon)^2+(t\epsilon)^2-\overline{\gamma_p(s\epsilon), \gamma_q(t\epsilon)}^2}{2st\epsilon^2}\right|=0
\end{equation}
for every $0<\tau_1<\tau_2<\infty$ and
\begin{equation}
\lim_{r \to 0}\frac{1}{\upsilon(B_r(x))}\int_{B_r(x)}|\cos \angle pzq-\cos \angle pxq|d\upsilon (z)=0.
\end{equation}
\end{remark}
\begin{remark}
Let $p, q \in M_{\infty}$.
Then \cite[Theorem $3.3$]{Ho} yields that there exists $A \subset M_{\infty}$ such that $\upsilon (M_{\infty}\setminus A)=0$ and that
\begin{equation}
\langle dr_p, dr_q \rangle (x)= \lim_{\delta \to 0}\frac{\overline{q, \gamma_p(\overline{x, p}+\delta)}- \overline{q, x}}{\delta}
\end{equation}
for every $x \in A$ and every $\gamma_p$.
Thus Theorem \ref{well3} and Lebesgue's differentiation theorem yield the following \textit{almost first variation formula for distance function} (in some weak sense):
\begin{equation}
\overline{q, \gamma_p(\overline{x, p}+\delta)}=\overline{q, x} + \delta \cos \angle pxq + o(\delta)
\end{equation}
for a.e. $x \in A$ and every $\gamma_p$, $\gamma_q$.
\end{remark}
The next corollary is a direct consequence of Corollary \ref{hol1} and \cite[Theorem $3.2$]{ho}:
\begin{corollary}[Weak H$\ddot{\mathrm{o}}$lder continuity of angles]\label{hold}
Let $\tau>0, R>1$, $p, q, x \in M_{\infty}$ with $p, q \in B_R(x) \setminus B_{R^{-1}}(x)$ and $x \in \mathcal{D}_p^{\tau}\cap \mathcal{D}_q^{\tau}$.
Then there exists $r=r(n, \tau, R)>0$ such that a function $\Phi(z) = \cos \angle pzq$ is $\alpha (n)$-H$\ddot{o}$lder continuous on $B_r(x) \cap \mathcal{D}_p^{\tau}\cap \mathcal{D}_q^{\tau}$.
In particular,  $\Phi$ is weakly $\alpha(n)$-H$\ddot{o}$lder continuous on $M_{\infty}$
with respect to any limit measure. 
\end{corollary}
Th next corollary is a direct consequence of Proposition \ref{we2nd} and Corollary \ref{hold}:
\begin{corollary}\label{holder}
Let $\{(C_i^l, \phi_i^l)\}_{l,i}$ be a rectifiable coordinate system constructed by distance functions on $M_{\infty}$ as in Theorem \ref{dist3}.
Then $M_{\infty}$ has a  weakly $C^{1, \alpha(n)}$-structure with respect to $\{(C_i^l, \phi_i^l)\}_{l,i}$.
Moreover for every $p \in M_{\infty}$, $dr_p$ is a weakly $\alpha(n)$-H$\ddot{o}$lder continuous $1$-form on $M_{\infty}$ with respect to $\{(C_i^l, \phi_i^l)\}_{l,i}$.
\end{corollary}
\begin{remark}
Let $(Y, p)$ be a Colding-Naber's example as in Theorem \ref{cn}. 
According to \cite[Theorem $1.2$]{co-na} by Colding-Naber, we see that $Y$ does \textit{NOT} have $C^{1, \beta}$-structure in ordinary sense for \textit{any} $0<\beta\le1$.
\end{remark}
Next, we will discuss the continuity of angles with respect to the Gromov-Hausdorff topology.
Recall that a map $\phi$ from a metric space $X_1$ to a metric space $X_2$ is said to be an \textit{$\epsilon$-Gromov-Hausdorff approximation} if $X_2 \subset B_{\epsilon}(\mathrm{Image} (\phi))$ and $|\overline{x, y}-\overline{\phi(x), \phi(y)}|<\epsilon$ for every $x, y \in X_1$.
\begin{proposition}\label{conti}
Let $(Y, y)$ be a Ricci limit space, $R>1$, $0<\tau <1$, $0<\beta <1$ and $p, q \in B_R(y)$ with $y \in M_{\infty} \setminus (C_p \cup C_q \cup \{p, q\})$.
Then for every $\epsilon>0$, there exists $\delta >0$ such that the following property holds:
Let $(\hat{Y}, \hat{y})$ be a Ricci limit space and $\hat{p}, \hat{q} \in B_R(\hat{y})$ with $\hat{y}\in \mathcal{D}_{\hat{p}}^{\tau} \cap \mathcal{D}_{\hat{q}}^{\tau} \setminus (B_{\beta}(\hat{p}) \cup B_{\beta}(\hat{q}))$.
Assume that there exists a $\delta$-Gromov-Hausdorff approximation $\phi$ from $(B_R(\hat{y}), \hat{y})$ to $(B_R(y), y)$ such that $\overline{\phi(\hat{p}), p}<\delta$, $\overline{\phi (\hat{q}), q}<\delta$. Then we have
$|\angle pyq - \angle \hat{p}\hat{y}\hat{q}|<\epsilon$.
\end{proposition}
\begin{proof}
The proof is done by contradiction.
Suppose that the assertion is false.
Then  there exist $\epsilon_0>0$, $R>1$, $\tau>0$, $\beta>0$, sequences of Ricci limit spaces $\{(Y_i, y_i)\}_{i<\infty}$ and of points $p_i, q_i \in B_R(y_i)$ such that $(B_R(y_i), y_i) \to (B_R(y), y)$, $p_i \to p, q_i \to q$, $y_i \in \mathcal{D}_{p_i}^{\tau} \cap \mathcal{D}_{q_i}^{\tau} \setminus (B_{\beta}(p_i) \cup B_{\beta}(q_i))$ for every $i$ and that $|\cos \angle p_iy_iq_i-\cos \angle pyq|\ge \epsilon_0$ for every $i$.
Moreover, by Gromov's compactness theorem, without loss of generality, we can assume that there exist a limit measure $\upsilon$ on $Y$ and 
a sequence $\{\upsilon_i\}_i$ of limit measures $\upsilon_i$ on $Y_i$ such that $\upsilon$ is the limit measure of $\{\upsilon_i\}_i$.
By Corollary \ref{hol1}, there exists $r>0$ such that
\begin{equation}
\left| \cos \angle p_iy_iq_i-\frac{1}{\upsilon_i (B_r(y_i))}\int_{B_r(y_i)}\langle dr_{p_i}, dr_{q_i}\rangle d\upsilon_i \right| + \left| \cos \angle pyq -\frac{1}{\upsilon (B_r(y))}\int_{B_r(y)}\langle dr_p, dr_q \rangle d\upsilon \right| < \frac{\epsilon_0}{3}
\end{equation}
for every $i$.
On the other hand, since $dr_{p_i} \to dr_p$ and $dr_{q_i} \to dr_q$ on $Y$, we have
\begin{equation}
\left| \frac{1}{\upsilon_i (B_r(y_i))}\int_{B_r(y_i)}\langle dr_{p_i}, dr_{q_i}\rangle d\upsilon_i-\frac{1}{\upsilon (B_r(y))}\int_{B_r(y)}\langle dr_p, dr_q \rangle d\upsilon \right| <\frac{\epsilon_0}{3}
\end{equation}
for every sufficiently large $i$.
Thus we have $|\cos \angle p_iy_iq_i-\cos \angle pyq|<\epsilon_0$ for every sufficiently large $i$.
This is a contradiction.
\end{proof}
The following theorem is  about the continuity of angles with respect to the Gromov-Hausdorff topology:
\begin{theorem}[GH-continuity of angles]
Let $R>1$, $\beta>0$ and $0<\tau<1$. Then for every $\epsilon>0$, there exists $\delta=\delta(n, R, \tau, \beta, \epsilon)>0$ such that the following property holds:
Let $(Y_1, y_1)$ and $(Y_2, y_2)$ be Ricci limit spaces, and $a_i, b_i \in B_R(y_i)$ with $y_i \in \mathcal{D}_{a_i}^{\tau} \cap \mathcal{D}_{b_i}^{\tau} \setminus (B_{\beta}(a_i) \cup B_{\beta}(b_i))$ for every $i=1, 2$.
Assume that there exists a $\delta$-Gromov-Hausdorff approximation $\phi$ from $(B_{R}(y_1), y_1)$ to $(B_{R}(y_2), y_2)$ such that $\overline{\phi(a_1), a_2}<\delta$ and $\overline{\phi(b_1), b_2}<\delta$.
Then we have $|\angle a_1y_1b_1-\angle a_2y_2b_2|< \epsilon$.
\end{theorem}
\begin{proof}
The proof is done by contradiction.
Suppose that the assertion is false.
Then by Gromov's compactness theorem, there exist 
\begin{enumerate}
\item $R>1$, $\beta>0$, $0< \tau<1$, $\epsilon_0>0$,
\item a Ricci limit space $(Z, z)$, points $a, b \in Z$,
\item a sequence of Ricci limit spaces $\{(Z_i^j, z_i^j)\}_{1 \le i<\infty, j=1,2}$,
\item a sequence of positive numbers $\{\delta_i\}_i$ with $\delta_i \to 0$,
\item sequences of points $a_i^j, b_i^j \in Z_i^j$ with $z_i^j \in \mathcal{D}_{a_i^j}^{\tau} \cap \mathcal{D}_{b_i^j}^{\tau} \cap (B_{R}(a_i^j) \setminus B_{\beta}(a_i^j)) \cap (B_{R}(b_i^j) \setminus B_{\beta}(b_i^j))$,
\item a sequence of $\delta_i$-Gromov-Hausdorff approximations $\phi_i$ from $(B_R(z_i^1), z_i^1)$ to $(B_R(z_i^2), z_i^2)$ with
$\overline{\phi_i(a_i^1), a_i^2}<\delta_i$ and $\overline{\phi_i(b_i^1), b_i^2}<\delta_i$,
\end{enumerate}
such that $(B_R(z_i^j), z_i^j) \to (B_R(z), z)$, $a_i^j \to a$, $b_i^j \to b$ as $i \to \infty$ for every $j=1,2$ and that $|\angle a_i^1z_i^1b_i^1-\angle a_i^2z_i^2b_i^2|\ge \epsilon_0$.
On the other hand, by Proposition \ref{conti}, we have $\lim_{i \to \infty}\angle a_i^jz_i^jb_i^j=\angle azb$. This is a contradiction.
\end{proof}
We end this subsection by giving an application of the weak H$\ddot{\mathrm{o}}$lder continuity of angles to a bi-Lipschitz embedding from a subset of $M_{\infty}$ to a Euclidean space.
Let $(\mathcal{R}_k)_{\delta, r}=\{x \in M_{\infty};d_{GH}((\overline{B}_t(x), x), (\overline{B}_t(0_k), 0_k))<\delta t$ for every $0<t<r$ $\}$, where $0_k \in \mathbf{R}^k$, and $d_{GH}$ is the Gromov-Hausdorff distance between pointed metric spaces.
See also \cite{ch-co1, ch-co3, co-na, gr} by Cheeger-Colding, Colding-Naber and Gromov.
\begin{proposition}
Let $R>1$, $r>0$, $\delta >0$, $\tau>0$ and $x \in (\mathcal{R}_k)_{\delta, r}$.
Assume that there exists $\{p_i\}_{1 \le i \le k} \subset M_{\infty}$ such that $x \in \bigcap _i((B_R(p_i) \setminus B_{R^{-1}}(p_i)) \cap \mathcal{D}_{p_i}^{\tau})$ and $\mathrm{det}(\cos \angle p_ixp_j)_{ij} \neq 0$.
Then the map $\phi_t =(r_{p_1}, \ldots, r_{p_k})\sqrt{(\cos \angle p_ixp_j)_{ij}}^{-1}$ from $B_t(x) \cap (\mathcal{R}_k)_{\delta, r} \cap \bigcap_i\mathcal{D}_{p_i}^{\tau}$ to $\mathbf{R}^k$ is an $(1 \pm \Psi(\delta, t; R, \beta, \tau, r))$-bi-Lipschitz embedding for every $0<t<r$, where $\Psi(a, b; c, d, e, f)$ is a positive definite function on $\mathbf{R}^6$ satisfying that $\lim_{a \to 0, b \to 0}\Psi(a, b; c, d, e, f)=0$ for every fixed $c, d, e, f$.
\end{proposition}
\begin{proof}
Let $\upsilon$ be a limit measure on $M_{\infty}$.
An argument similar to that of the proof of Theorem \ref{well3} yields
\begin{equation}
\lim_{r \to 0}\frac{1}{\upsilon (B_r(x))}\int_{B_r(x)}\mathrm{det}(\langle dr_{p_i}, dr_{p_j}\rangle )_{ij}d\upsilon = \det(\cos \angle p_ixp_j)_{ij}.
\end{equation}
Then the proposition follows from an argument similar to that of the proof of \cite[Lemma $3.14$]{Ho}.
\end{proof}
\begin{remark}
Assume $(M_{\infty}, m_{\infty})$ is a \textit{noncollapsing limit}, i.e., there exists a sequence of pointed $n$-dimensional complete Riemannian manifolds $\{(M_i, m_i)\}_i$ with $\mathrm{Ric}_{M_i} \ge -(n-1)$ such that  $\lim_{i \to \infty}\mathrm{vol}\,B_1(m_i) >0$ and
$(M_i, m_i) \to (M_{\infty}, m_{\infty})$.
Then in \cite[Theorem $5.11$]{ch-co}, Cheeger-Colding showed that for every $x \in (\mathcal{R}_n)_{\delta, r}$, we have $B_{r/32}(x) \subset (\mathcal{R}_n)_{\Psi(\delta, r;n), r/32}$.
See also \cite[Remark $5.15$]{ch-co1}, \cite[Theorem B.2]{co-na1} and \cite[Theorem $1.1$]{co-na} by Cheeger-Colding and Colding-Naber for related results.
\end{remark}
\subsection{Weak Lipschitz continuity of the Riemannian metric on a Ricci limit space}
In this subsection, we will show that the Riemannian metric of a Ricci limit space is weakly Lipschitz.

Assume that  $(M_i, m_i, \mathrm{vol}/\mathrm{vol}\,B_1(m_i)) \to (M_{\infty}, m_{\infty}, \upsilon)$.
The following proposition is an essential result to get Theorem \ref{limit}.
See \cite{Kasue1, Kasue2, ka-ku1, ka-ku2} by Kasue and Kasue-Kumura for related important interesting results.
\begin{proposition}\label{ppp}
Let $R>0$ and let $\{f_i\}_{1\le i \le \infty}$ be a sequence of Lipschitz functions $f_i$ on $B_R(m_i)$ with $\sup_i \mathbf{Lip}f_i<\infty$.
Assume that the following hold:
\begin{enumerate}
\item $(f_i, df_i) \to (f_{\infty}, df_{\infty})$ on $B_R(m_{\infty})$.
\item There exists $r>0$ with $r<R$ such that $\mathrm{supp} (f_i) \subset B_r(m_i)$ for every $i$.
\item $|df_i|^2 \in H_{1,2}(B_R(m_i))$ for every $i<\infty$, and
\begin{equation}
\sup_{i<\infty} \frac{1}{\mathrm{vol}\,B_R(m_i)}\int_{B_R(m_i)}|d|df_i|^2|^2d\mathrm{vol}<\infty.
\end{equation}
\end{enumerate}
Then we have $|df_{\infty}|^2 \in H_{1,2}(B_R(m_{\infty}))$ and 
\begin{equation}
\frac{1}{\upsilon (B_R(m_{\infty}))}\int_{B_R(m_{\infty})}|d|df_{\infty}|^2|^2d\upsilon \le \liminf_{i \to \infty}\frac{1}{\mathrm{vol}\,B_R(m_i)}\int_{B_R(m_i)}|d|df_i|^2|^2d\mathrm{vol}.
\end{equation}
\end{proposition}
\begin{proof}
\cite[Lemma $5.8$]{KS} by Kuwae-Shioya (or \cite[Lemma $5.17$]{di2} by Ding) yields that there exists an orthonornal basis $\{\phi_i^j\}_j$ on $L^2(B_R(m_i))$ consisting of eigenfunctions $\phi_i^j$ associate with the $j$-th eigenvalue $\lambda_i^j$ with respect to the Dirichlet problem on $B_R(m_i)$ such that $\lambda_i^j \to \lambda_{\infty}^j$ and that $\phi_i^j \to \phi_{\infty}^j$ with respect to the $L^2$-topology (see \cite[Definition $2.3$]{KS} by Kuwae-Shioya for the definition of $L^2$-topology with respect to the measured Gromov-Hausdorff topology, or \cite{Ho5}).
Put
$|df_i|^2=\sum_{j=0}^{\infty}a_i^j\phi_i^j$
in $L^{2}(B_R(m_i))$ for every $i\le \infty$.
Let $L \ge 1$ with
\begin{equation}
\frac{1}{\mathrm{vol}\,B_R(m_i)}\int_{B_R(m_i)}|d|df_i|^2|^2d\mathrm{vol} = \sum_{j=0}^{\infty}\lambda_i^j(a_i^j)^2 \le L 
\end{equation}
for every $i<\infty$.
\cite[Lemma $5.11$]{di2} by Ding yields  
\begin{equation}
\sum_{j=N+1}^{\infty}(a_i^j)^2 \le \frac{1}{(\lambda_i^{N+1})^{1/2}}||f_i||_{L^2(B_R(m_i))}||d|df_i|^2||_{L^2(B_R(m_i))}\le \frac{C(n, R, L)}{N^{1/n}}
\end{equation}
for every $i < \infty$ and every $N$.
Fix $\epsilon>0$.
Then there exists $N_0$ such that 
$\sum_{j=N_0+1}^{\infty}(a_i^j)^2<\epsilon$
for every $i < \infty$.
Since $|df_i|^2 \to |df_{\infty}|^2$ on $B_R(m_{\infty})$ with respect to the $L^2$-topology,
we have 
\begin{equation}
a_i^j = \frac{1}{\mathrm{vol}\,B_R(m_i)}\int_{B_R(m_i)}|df_i|^2\phi_i^j d\mathrm{vol} 
\stackrel{i \to \infty}{\to} \frac{1}{\upsilon(B_R(m_{\infty}))}\int_{B_R(m_{\infty})}|df_{\infty}|^2\phi_{\infty}^j d\upsilon=a_{\infty}^j.
\end{equation}
Thus we have 
$|||df_{\infty}|^2-\sum_{j=0}^Na_{\infty}^j\phi_{\infty}^j||_{L^2(B_R(m_{\infty}))}=\lim_{i \to \infty }
|||df_{i}|^2-\sum_{j=0}^Na_{i}^j\phi_{i}^j||_{L^2(B_R(m_{i}))}\le \epsilon,$ for every $N \ge N_0$, 
i.e., $\sum_{j=0}^Na_{\infty}^j\phi_{\infty}^j \to |df_{\infty}|^2$
in $L^2(B_R(m_{\infty}))$ as $N \to \infty$.
Since
$||d(\sum_{j=0}^Na_i^j\phi_i^j)||_{L^2(B_R(m_i))} \to ||d(\sum_{j=0}^Na_{\infty}^j\phi_{\infty}^j)||_{L^2(B_R(m_{\infty}))}$
as $i \to \infty$ 
for every $N$,
this completes the proof.
\end{proof}
\begin{corollary}\label{sob}
Let $R>0$, $L\ge 1$ and let $\{f_i\}_i$ be a sequence of Lipschitz functions $f_i$ on $B_R(m_i)$.
Assume that the following properties hold:
\begin{enumerate}
\item $f_i$ is a $C^2$-function for every $i<\infty$.
\item \[\sup_{i<\infty} \left(||f_i||_{L^{\infty}} + \mathbf{Lip}f_i + \frac{1}{\mathrm{vol}\,B_R(m_i)}\int_{B_R(m_i)}(\Delta f_i)^2d\mathrm{vol}\right) \le L.\]
\item $f_i \to f_{\infty}$ on $B_R(m_{\infty})$.
\end{enumerate}
Then we have $|df_{\infty}|^2 \in H_{1,2}(B_r(m_{\infty}))$ for every $r<R$, and 
\begin{equation}
\frac{1}{\upsilon(B_r(m_{\infty}))}\int_{B_r(m_{\infty})}|d|df_{\infty}|^2|^2 d\upsilon \le C(n, L, r, R).
\end{equation}
In particular, we see that $|df_{\infty}|^2$ is weakly Lipschitz on $B_R(m_{\infty})$.
\end{corollary}
\begin{proof}
The existence of a good cutoff function \cite[Theorem $6.33$]{ch-co} by Cheeger-Colding yields that there exists a sequence $\{\phi_i\}_{i<\infty}$ of smooth functions $\phi_i$ on $B_R(m_i)$ such that $||\nabla \phi_i||_{L^{\infty}} \le C(n, r, R)$, $||\Delta \phi_i||_{L^{\infty}} \le C(n, r, R)$, $0 \le \phi \le 1$, $\phi_i|_{B_r(m_i)}\equiv 1$ and $\mathrm{supp} (\phi_i) \subset B_{(r+R)/2}(m_i)$.
By applying Proposition \ref{ppp} for $\phi_if_i$, the property of $(2)$ in page $8$ and \cite[Remark $4.2$]{Ho}, it follows that
\begin{equation}
\frac{1}{\upsilon(B_r(m_{\infty}))}\int_{B_r(m_{\infty})}|d|df_{\infty}|^2|^2 d\upsilon \le C(n, L, r, R).
\end{equation}
On the other hand, Cheeger-Colding proved in \cite[Theorem $2.15$]{ch-co3} that the Poincar\'e inequality of type $(1,2)$ on $M_{\infty}$ holds.
Thus \cite[Theorem $4.14$]{ch1} by Cheeger yields that any Sobolev function is weakly Lipschitz.
Therefore we have the corollary.
\end{proof}
The following is a direct consequence of Corollary \ref{sob}:
\begin{theorem}[Weak twice differentiability of Ricci limit spaces]\label{harm}
Let $\{(C_i^l, \phi_i^l)\}_{l, i}$ be a rectifiable coordinate system of $(M_{\infty}, \upsilon)$.
Assume that for every $i, l$, there exist 
 $r>0$,  a sequence $\{x_j\}_j$ of points $x_j \in M_{j}$ with $C_i^l \subset B_r(x_{\infty})$ and $x_j \to x_{\infty}$,
 a sequence $\{f_{j, s}\}_{j<\infty, 1 \le s \le l}$ of $C^2$-functions $f_{j, s}$ on $B_{r}(x_j)$ such that
$\sup_{j, s}\mathbf{Lip}f_{j,s}<\infty$, 
$f_{j,s} \to \phi_{i, s}^{l}$ on $C_i^l$ as $j \to \infty$ for every $s$ and that
\begin{equation}
\sup_{j, s}\frac{1}{\mathrm{vol}\,B_r(x_j)}\int_{B_r(x_j)}(\Delta f_{j,s})^2d\mathrm{vol}<\infty,
\end{equation}
where $\phi_i^l=(\phi_{i,1}^{l}, \ldots \phi_{i, l}^{l})$.
Then the Riemannian metric $g$ of $M_{\infty}$ is weakly Lipschitz with respect to $\{(C_i^l, \phi_i^l)\}_{l, i}$.
In particular,  $M_{\infty}$ has a weakly second order differential structure with respect to $\{(C_i^l, \phi_i^l)\}_{l, i}$.
\end{theorem}
We now are in a position to prove Theorem \ref{limit}:

\textit{A proof of Theorem \ref{limit}.}

It follows directly from Theorems \ref{harm3} and \ref{harm}. \,\,\,\,\, $\Box$
\begin{definition}
We say that a rectifiable coordinate system $\{(C_i^l, \phi_i^l)\}_{l, i}$ of $(M_{\infty}, \upsilon)$ as in Corollary \ref{harm} is a \textit{weakly second order differential structure associated with} 

$\{(M_j, m_j, \mathrm{vol}/\mathrm{vol}\,B_1(m_j))\}_j$.
\end{definition}
Assume that $\{(C_i^l, \phi_i^l)\}_{l, i}$ is a weakly second order differential structure associated with $\{(M_j, m_j, \mathrm{vol}/\mathrm{vol}\,B_1(m_j))\}_j$.
\begin{proposition}\label{2nd2nd}
Let $R>0$ and let $f_{\infty}$ be a Lipschitz function on $B_R(m_{\infty})$.
Assume that there exists a sequence $\{f_j\}_{j<\infty}$ of $C^2$-functions $f_j$ on $B_R(m_j)$ such that $\sup_j\mathbf{Lip}f_j<\infty$, $f_j \to f_{\infty}$ on $B_{R}(m_{\infty})$ and 
\begin{equation}
\sup_{j<\infty}\frac{1}{\mathrm{vol}\,B_R(m_j)}\int_{B_R(m_j)}(\Delta f_{j})^2d\mathrm{vol}<\infty.
\end{equation}
Then $f_{\infty}$ is weakly twice differentiable on $B_R(m_{\infty})$ with respect to $\{(C_i^l, \phi_i^l)\}_{l, i}$.
\end{proposition} 
\begin{proof}
The proposition follows from Corolalry \ref{sob}.
\end{proof}
Finally, we end this section by giving the following corollary:
\begin{corollary}[Weak twice differentiability of eigenfunctions]\label{eigen}
Let $f_{\infty}$ be an eigenfunction associated with the eigenvalue $\lambda_{\infty}$ with respect to the Dirichlet problem on $B_R(m_{\infty})$.
Then $f_{\infty}$ is weakly twice differentiable on $B_R(m_{\infty})$ with respect to $\{(C_i^l, \phi_i^l)\}_{l, i}$. 
\end{corollary}
\begin{proof}
\cite[Lemma $5.8$]{KS} by Kuwae-Shioya (or \cite[Lemma $5.17$]{di2} by Ding) yields that there exists a sequence $\{f_i\}_i$ of eigenfunctions $f_i$ associated with the eigenvalue $\lambda_i$ with respect to the Dirichlet problem on $B_R(x_i)$ such that $\lambda_i \to \lambda_{\infty}$ and that $f_i \to f_{\infty}$ with respect to the $L^2$-topology.
Note that it follows from Cheng-Yau's gradient estimate \cite{ch-yau} that $\sup_{i}\mathbf{Lip}(f_i|_{B_r(x_i)})<\infty$ for every $r<R$.
Thus the corollary follows directly from Proposition \ref{2nd2nd}.
\end{proof}
\begin{remark}
See \cite[Theorem $1.3$]{Ho5} for a generalization of Corollary \ref{sob} and Proposition \ref{2nd2nd}.
Moreover, in \cite{Ho5}, we will prove that for $f_{\infty}$ as in Corollary \ref{eigen}, if $M_{\infty}$ is noncollapsing, then $\Delta^{g_{M_{\infty}}}f_{\infty}=\lambda_{\infty}f_{\infty}$.
In particular, in noncollapsing setting, the Laplacian defined as in Proposition \ref{hess2} coincides with the Dirichlet Laplacian defined by Cheeger-Colding in \cite{ch-co3} on a dense subspace in $L^2$.
See \cite[Theorem $1.4$]{Ho5} for the detail.
\end{remark}

\end{document}